\documentclass{amsart}

\usepackage{amsmath,amsthm,amsfonts,amssymb}
\usepackage{framed}
\usepackage{nicefrac}
\usepackage{color,graphicx}
\usepackage[x11names, rgb]{xcolor}
\usepackage{tikz}
\usetikzlibrary{cd}
\usepackage{hyperref}
\usepackage{pdfsync}
\usepackage{enumitem}
\usepackage[capitalize,nameinlink]{cleveref}
\crefname{conjecture}{Conjecture}{Conjectures}
\Crefname{conjecture}{Conjecture}{Conjectures}
\usepackage[utf8]{inputenc}
\usepackage{tableau,ytableau}
\usepackage[position=bottom]{subfig}

\theoremstyle{plain}
\newtheorem{theorem}{Theorem}
\newtheorem{lemma}[theorem]{Lemma}
\newtheorem{conjecture}[theorem]{Conjecture}
\newtheorem*{theorem*}{Theorem}
\newtheorem{proposition}[theorem]{Proposition}
\newtheorem{corollary}[theorem]{Corollary}

\crefname{mainTheorem}{Theorem}{Theorems}

\theoremstyle{definition}

\theoremstyle{remark}
\newtheorem{example}[theorem]{Example}

\newtheorem{remark}[theorem]{Remark}
\newtheorem*{remark*}{Remark}

\numberwithin{theorem}{section}

\def\u{\ensuremath{\mathfrak u}}
\def\C{\ensuremath{\mathbb C}}

\def\P{\ensuremath{\mathbb P}}
\newcommand{\cO}{\mathcal{O}}
\newcommand{\one}{1\hskip-3.5pt1}
\newcommand{\csm}{{c_{\text{SM}}}}
\newcommand{\csmT}{{c_{\text{SM}}^T}}
\newcommand{\csmv}{{c^\vee_{\text{SM}}}}

\DeclareMathOperator{\Fl}{Fl}

\newcommand{\cMa}{{c_{\text{Ma}}}}

\newcommand{\cMaT}{{c_{\text{Ma}}^{T}}}
\newcommand{\cMaTv}{{c^{T,\vee}_{\text{Ma}}}}

\newcommand{\al}{\alpha}

\newcommand{\bP}{{\mathbb{P}}}

\newcommand{\ssm}{{s_{\text{SM}}}}

\newcommand{\cma}{{c_{\text{Ma}}}}
\newcommand{\sma}{{s_{\text{Ma}}}}

\DeclareMathOperator{\LG}{{LG}}
\DeclareMathOperator{\OG}{{OG}}

\DeclareMathOperator{\Gr}{Gr}

\DeclareMathOperator{\Eu}{Eu}

\newcommand{\rp}{\ensuremath{R_{\geq\alpha_P}}}

\begin{document}

\title[Mather classes and conormal spaces]{Mather classes and conormal spaces of Schubert varieties in cominuscule spaces}
\author{Leonardo C.  Mihalcea}
\address{Department of Mathematics, Virginia Tech, 460 McBryde Hall, 225 Stanger St., Blacksburg VA 24061}
\email{lmihalce@math.vt.edu} 
\author{Rahul Singh}
\address{Department of Mathematics, Virginia Tech, 460 McBryde Hall, 225 Stanger St., Blacksburg VA 24061}
\email{rahul.sharpeye@gmail.com}


\thanks{L.~C.~Mihalcea was supported in part by the Simons Collaboration Grant 581675.}

\date{June 8, 2020} 
\maketitle

\begin{abstract}
Let $G/P$ be a complex cominuscule flag manifold.
We prove a type independent formula for the torus equivariant Mather class of a 
Schubert variety in $G/P$, and for a Schubert variety 
pulled back via the natural projection $G/Q \to G/P$. We apply this to find formulae for the local Euler obstructions 
of Schubert varieties, and for the torus equivariant localizations of the conormal spaces 
of these Schubert varieties. We conjecture positivity properties for the local Euler obstructions and 
for the Schubert expansion of Mather classes. We check the conjectures in many cases, 
by utilizing results of Boe and Fu about the characteristic cycles of the 
intersection homology sheaves of Schubert varieties.
We also conjecture that certain `Mather polynomials' are unimodal in general Lie type, and log concave in type A.
\end{abstract}

\section{Introduction}\label{sec:intro}
Let $X$ be a complex, projective manifold and let $Y \subset X$ be a closed irreducible subvariety.
The Mather class $\cMa(Y)$ is a non-homogeneous element in the (Chow) homology $A_*(X)$.
Its original definition uses the Nash blowup of $X$ along $Y$,
but in this paper we work with the following equivalent definition, going back to Sabbah \cite{sabbah:quelques};
see also \cite{ginzburg:characteristic,AMSS:shadows}.

Let $T^*(X)$ be the cotangent bundle of $X$, and let $\iota:X \to T^*(X)$ be the zero section embedding.
The multiplicative group $\C^*$ acts on $T^*(X)$ by fibrewise dilation with character $\hbar^{-1}$.
To the subvariety $Y$ one associates the conormal space $T^*_Y(X) \subset T^*(X)$; 
this is an irreducible conic Lagrangian cycle in the cotangent bundle.
The Mather class $\cMa(Y)$ is the dehomogenization of the $\C^*$-equivariant class of the conormal space: 
\[ \cMa(Y):= (-1)^{\dim Y} (\iota^*[T^*_Y(X)]_{\C^*})_{\hbar =1}  \quad \in A_0^{\C^*}(X) \/. \] 
For example, it follows from definition that if $Y$ is smooth, then $\cMa(Y)$ is the push-forward 
of the homology class $c(T_Y) \cap [Y]$ inside $A_*(X)$.
If $Y=X$, then one recovers the well known index 
formula for the topological Euler characteristic:
\[ \chi(X) = (-1)^{\dim X} \int_X \iota^*[T^*_X(X)] \/.\] 
An equivariant version of Mather classes was defined by Ohmoto 
\cite{ohmoto:eqcsm}; we refer to \cref{sec:segre-csm} below for the precise details.

Let $G$ be a complex, semisimple Lie group, and fix $T \subset B \subset P \subset G$ a parabolic subgroup 
$P$ containing a standard Borel subgroup $B$ with a maximal torus $T$;
let $X=G/P$ be the associated flag manifold.
The goal of this paper is to study the $T$-equivariant Mather class
$\cMaT(Y) \in H_0^{T \times \C^*}(X)$ for a Schubert variety $Y$ in a cominuscule space $G/P$,
or when $Y$ is a Schubert variety in an arbitrary flag manifold $G/Q$ obtained by pulling back via the natural projection $G/Q \to G/P$.

The cominuscule spaces are a family of flag manifolds consisting of the ordinary Grassmannian, 
the maximal orthogonal Grassmannians in Lie types B,D, the Lagrangian Grassmannian in type C, quadrics, 
and respectively the Cayley plane and the Freudenthal variety 
in the exceptional Lie types $E_6$ and $E_7$.
\begin{footnote}
{
For the cominuscule property to hold, the maximal orthogonal
Grassmannian in type B needs to be regarded as a homogeneous space under the Lie group of type D;
see \cref{sec:preliminaries} below.
}
\end{footnote} 

Let $W$ denote the Weyl group and let $W^P$ be the subset of minimal length representatives.
For $w \in W^P$, let $X_w^{P,\circ}=BwP/P$ be the Schubert cell in $G/P$, and let $X_w^P:=\overline{X_w^{P,\circ}}$ be the Schubert variety;
let also $X_w^B := \overline{BwB/B}$ be the Schubert variety in $G/B$.

The Mather class of a Schubert variety is related to Chern-Schwartz-MacPherson (CSM) classes of its Schubert cells via the {\em local Euler obstruction} coefficients $e_{w,v}$:
\begin{equation}
\label{E:Ma=CSM} \cMa(X_w^P) = \sum_v e_{w,v} \csm(X_v^{P,\circ}) \/.
\end{equation}
These coefficients were defined by MacPherson \cite{macpherson:chern} and provide a subtle measure of the singularity of $X_w^P$ at $v$.
For instance, consider the parabolic Kazhdan-Lusztig  (KL) polynomial $P_{w,v}(q)$;
cf.~\cite{deodhar:geometricII}.
Then the equalities
\begin{equation}
\label{E:eP} e_{w,v} = P_{w,v}(1), \quad \forall v \in W^P \/,
\end{equation}
hold if and only if the characteristic cycle of the intersection homology (IH) sheaf of the Schubert variety $X_w^P$ is irreducible.
In general, the problem of finding the decomposition of the characteristic cycle of the IH sheaves into irreducible components is open, although some particular cases are known;
see e.g.~\cite{KL:topological,kashiwara.tanisaki:characteristic,MR1084458,boe.fu,evens.mirkovic:characteristic,braden:irred,williamson:reducible},
and also \cref{sec:positivity} below for more details.
We note that since CSM classes of Schubert cells can be explicitly calculated \cite{aluffi.mihalcea:csm,aluffi.mihalcea:eqcsm,rimanyi.varchenko:csm}, equation \eqref{E:Ma=CSM} shows that giving an algorithm to calculate Mather classes is equivalent to one for the local Euler obstructions.

To state a precise version of our results, we need to introduce more notation.
For $w \in W^P$, let $I(w)$ denote the inversion set of $w$;
this consists of positive roots $\alpha$ such that $w(\alpha) <0$,
and it may be identified with the {\em diagram} of $w$.
For a root $\alpha$ we denote by $\mathfrak{g}_\alpha$ the root subspace of $Lie(G)$ 
determined by $\alpha$ and by $\C_\alpha$ the one-dimensional $B$-module of weight $\alpha$.
It follows from \cite{singh:conormalII} (see also \cite{richmond.slofstra.woo}) that if $G/P$ is cominuscule and 
$w \in W^P$ then the vector space
$T_w := \oplus_{\alpha \in I(w)} \mathfrak{g}_{-\alpha}$
has a structure of a $B$-module.
Therefore 
\[ \mathcal{T}_w:= G \times^B T_w\]
is a vector bundle over the complete flag variety $G/B$.
Let $c(\mathcal{T}_w)$ denote its total Chern class.

The following is the main result of our paper;
see \cref{thm:mather,thm:pbMa} below.

\begin{theorem}
\label{thm:mainintro}
Let $G/P$ be a cominuscule space and projection $\pi:G/B \to G/P$, and let $w \in W^P$ be a minimal length representative.
Then the following hold:
\begin{enumerate}[label=(\alph*)]

\item
The Mather class of $X_w^P$ is given by 
\[ \cMa(X_w^P) = \pi_*( c(\mathcal{T}_w) \cap [X_w^B]) = \pi_* (\prod_{\alpha \in I(w)} c(G \times^B \C_{-\alpha}) \cap [X_w^B]) \/. \] 

\item
Let $Q \subset P$ be any parabolic subgroup, with $\pi_Q:G/Q \to G/P$ the natural projection.
Then the Mather class of the pull-back Schubert variety $\pi^{-1}(X_w^P)$ is
\[ \cMa(\pi_Q^{-1}(X_w^P)) = c(T_{\pi_Q}) \cap \pi_Q^*(\cMa(X_w^P))\/, \]
where $T_{\pi_Q}$ is the relative tangent bundle of the projection $\pi_Q$.

\item
The formulae in (a) and (b) hold in the $T$-equivariant setting.
\end{enumerate}
\end{theorem}
We encourage the reader to jump directly to section \ref{ss:examples} for examples illustrating the formula in part (a) and its equivariant version.

The proof of part (a) exploits the observation that the $\C^*$-equivariant pull-back 
$\iota^* [T^*_{X_w^P}(G/P)]$ is essentially given by the Segre class of the conormal space $T^*_{X_w^P}(G/P)$; 
see \cref{sec:segre-cones} below.
To calculate this Segre class, 
we utilize a desingularization of the conormal space found by the second named author \cite{singh:conormalII}, 
together with the property that the Segre classes are preserved under birational push forward.

A different proof of the part (a) of Theorem \ref{thm:mainintro} may be obtained using the identification by 
Richmond, Slofstra and Woo \cite[Thm. 2.1]{richmond.slofstra.woo} of the Nash blowup of the Schubert varieties in cominuscule spaces.
In this paper we aimed to emphasize the 
equivalence between Mather classes and the Segre classes of the conormal spaces,
a point of view which we believe it will have further benefits for understanding the conormal spaces.

Part (b) follows from the Verdier-Riemann-Roch formula proved by Yokura \cite{yokura:verdier},
and from the invariance of Euler obstructions under smooth pull-back.
The latter statement is likely well known to experts, but we could not find it in the form we need in the literature.
We give it two proofs, one under very general hypotheses in \cref{prop:eulerpb},
and the second in \cref{sec:matherpb} which uses pull-backs of conormal spaces.
All constructions are $T$-equivariant, and part (c) follows.

We give two applications of Theorem \ref{thm:mainintro}.
The first is an explicit localization formula for the conormal spaces $T^*_{X_w^P}(G/P)$ of Schubert varieties in cominuscule spaces;
see \cref{thm:conormalloc}.

The second application is a formula for the local Euler obstructions of Schubert varieties.
The proof uses the equation \eqref{E:Ma=CSM} and the identification of the Poincar{\'e} duals of CSM classes obtained in \cite{AMSS:shadows}.
The resulting formula is given in \cref{thm:eulerobs}.
Based on many calculations in all cominuscule types we conjecture the following positivity properties;
cf.~ \cref{conj:pos,conj:eulerpos} below.

\begin{conjecture}[Positivity Conjecture]
\label{conj:intro}
Let $X=G/P$ be a cominuscule space and let $v,w \in W^P$.
\begin{enumerate}[label=(\alph*)]
\item
Consider the Schubert expansion
\[ \cMa(X_w^P) = \sum_{v \le w} a_{w,v} [X_v^P] \/.\]
Then $a_{w,v} \ge 0$.
A positivity property also holds for the equivariant Mather classes (cf.~Conj. \ref{conj:pos}).
\item
The local Euler obstruction coefficients are non-negative, i.e. $e_{w,v} \ge 0$.
\end{enumerate}
\end{conjecture} 

By the equation \eqref{E:Ma=CSM} and positivity of the non-equivariant CSM classes of Schubert cells \cite{huh:csm,AMSS:shadows}, 
part (b) implies the non-equivariant positivity from part (a).
It is tempting to make this conjecture in {\em any} flag manifold $G/Q$,
but unfortunately we do not have substantial evidence in this generality.
Most of the other cases we can check follow from Proposition \ref{prop:eulerpb} below,
which states that the local Euler obstructions are preserved under smooth pull-backs;
therefore one may expand this conjecture to include pull-backs of Schubert varieties from $G/P$.

By the positivity of KL polynomials and the equation \eqref{E:eP}, $e_{w,v} >0$ whenever the characteristic cycle of the IH sheaf of $X_w^P$ is irreducible.
This holds for cominuscule spaces in Lie types A and D, by results from \cite{MR1084458} and \cite{boe.fu}.
Boe and Fu also prove positivity of Euler obstructions for the odd-dimensional quadrics (in Lie type B).
Therefore, Conjecture \ref{conj:intro} holds in all these cases.
See \cref{sec:positivity} below for more details.

We also conjecture a unimodality property for the {\em Mather polynomial} of $w \in W^P$.
The Mather polynomial is obtained from the Schubert expansion of the Mather class by replacing each Schubert class $[X_v^P]$ by $x^{\ell(v)}$.
We conjecture that the resulting polynomial is unimodal, in the sense of \cite{stanley:Log}.
For the ordinary Grassmannians, calculations suggest that the polynomial is also log concave;
see \cref{ss:logc} for details and examples.

Formulas for the Mather classes and for the local Euler obstructions have been found by B. Jones \cite{jones:csm} 
in the case of Grassmann manifolds, and in \cite{raicu:characters,zhang:chern,rimanyi.protampan,timchenko} for various types of degeneracy loci.
Jones' proof is based on the fact that if $\pi': Z_w \to X_w^P$ 
is a small resolution of $X_w^P$ (in the sense of intersection homology) \emph{and}
if the characteristic cycle of the IH sheaf of $X_w^P$ is irreducible, then the Mather class satisfies 
\[ \cMa(X_w^P) =  \pi'_*(c(T_{Z_w}) \cap [Z_w]) \/, \] 
where $c(T_{Z_w})$ is the total Chern class of the tangent bundle of $Z_w$.
Small resolutions for the Schubert varieties in
Grassmannians were constructed by Zelevinsky \cite{MR705051},
and Bressler, Finkelberg and Lunts \cite{MR1084458} proved that
the characteristic cycles of the IH sheaves of Schubert varieties are irreducible.
Outside the type A Grassmannian, Schubert varieties may not admit small resolutions; see 
\cite{sankaran.v:small-res} and also \cite[Example 7.15]{perrin:small-res}. 

Boe and Fu \cite{boe.fu} used delicate techniques from geometric analysis to find 
formulae for the local Euler obstruction $e_{w,v}$ of the Schubert varieties in cominuscule 
spaces $G/P$ of classical Lie types A--D.
Using recursive formulae for the KL polynomials,
they were able to show that the identities \eqref{E:eP} hold in Lie types A and D, and fail in general
for types B, C.
We included examples such as \cref{ex:EulerLG36,ex:LG24div},
recovering instances of reducible IH sheaves from \cite{boe.fu} and \cite{kashiwara.tanisaki:characteristic},
and obtained with the formulae from this paper.
In future work, we plan to compare our formula from 
Theorem \ref{thm:eulerobs} to the formulae in \cite{boe.fu}.

{\em Acknowledgements.} LM would like to thank P. Aluffi, J. Sch{\"u}rmann and C. Su for related collaborations, and to D. Anderson and E. Richmond for useful discussions.
RS would like to thank D. Muthiah for useful discussions.
To perform calculations in (equivariant) cohomology,
we utilized A. Buch's Maple program {\em Equivariant Schubert Calculator},
available at \texttt{https://sites.math.rutgers.edu/$\sim$asbuch/equivcalc/}

Throughout the paper we utilize the Chow (co)homology theory \cite{fulton:IT},
and its equivariant version from \cite{edidin.graham:eqchow}.
This is related to the ordinary (possibly equivariant) (co)homology via the cycle map  - see \cite[Ch.~19]{fulton:IT} and \cite[\S 2.8]{edidin.graham:eqchow}; 
for flag manifolds this map is an isomorphism \cite[Ex.~19.1.11]{fulton:IT}.
Finally, we work over the field of complex numbers.

\section{Segre classes of cones}
\label{sec:segre-cones}
\subsection{Segre classes and the pull back via the zero section}
The treatment in this section follows largely \cite[\S 4]{fulton:IT},
but we also used \cite[\S1]{BBM:book} and \cite{BBM:Springer}.
Let $C$ be a cone in the sense of \cite[\S 4 and Appendix B.5]{fulton:IT}.
For the applications envisioned in this note, we assume in addition that
$C$ is a closed subcone of a vector bundle $E \to X$.
We consider the projective completion $\bP(E \oplus \one)$,
and we denote by $\cO_{E}(-1)$ the tautological bundle of lines in $E\oplus\one$.
Denote by $\cO_C(-1)$ the restriction of $\cO_{E}(-1)$ to
the projective completion $\overline{C}=\bP(C \oplus \one)  \subset \bP(E \oplus \one)$,
and let $q: \bP(E \oplus \one) \to X$ be the natural projection.

The {\em Segre class} of $C$ is the (non-homogeneous) class in the Chow group $A_*(X)$ defined by:
\begin{equation}
\label{def:segre-class}
s(C):=q_*\left(\frac{[\overline{C}]}{c(\cO_C(-1))}\right)= q_*\left(\sum_{i \ge 0}c_1(\cO_C(1))^i \cap [\overline{C}]\right) \/.
\end{equation}
If $C=E$ is a vector bundle over $X$ then its Segre class is $s(E)=c(E)^{-1}\cap [X]$, see \cite[Prop.~4.1]{fulton:IT}.

Suppose now that the cone $C \subset E$ is pure dimensional, with $\dim C =\operatorname{rank}(E)$.
(This will be the case for our application, when $C$ is the conormal space of a subvariety.)
In this case, the Segre class of $C$ is related to the pull back $\iota^*[C]$ of the class of $C$ via the zero section $\iota:X \to E$.
Observe however that the Segre class is non-homogeneous, while $\iota^*[C]$ is a class in $A_0(X)$.
In order to relate the two, one needs to work in the $\C^*$-equivariant Chow group;
this was one of the observations in \cite[\S 2]{AMSS:shadows}.
We recall the relevant facts next,
referring the reader to \cite{edidin.graham:eqchow} for details on equivariant Chow groups.

The starting point is the formula from \cite[Example 4.1.8]{fulton:IT}, which states that
\begin{equation}
\label{E:iota}
\iota^*[C] = (c(E) \cap s(C) )_{0}\/,
\end{equation}
in the Chow group of $X$, where
$(a)_{0}$ means taking the homogeneous component of degree $0$ of the class $a \in A_*(X)$.
We will need to `homogenize' this formula.

There is a $\C^*$-action on $E$ by dilation by a character $\chi$, which extends to an action on
$E \oplus \one$ by letting $\C^*$ act trivially on the second component.
This induces a $\C^*$-action on the projective completion $\P(E \oplus \one)$.
Both $C$ and its closure are $\C^*$-stable subschemes; the action of $\C^*$ restricted to the base $X$ is trivial.
The character $\chi$ determines a class in the equivariant Chow group $A^1_{\C^*}(pt)$ of degree $1$, denoted in the same way.
Since $\C^*$ acts trivially on $X$,
a class $a \in  A_0^{\C^*}(X)$ is equivalent to a {\em non-homogeneous} class $a_0 + a_1 + \ldots \in A_*(X)$ ($a_i \in A_i(X)$) obtained by dehomogenizing $a$.
Conversely, if $a=a_0 + a_1 + \ldots \in A_*(X)$ is a non-homogeneous class, its {\em $\chi$-homogenization} is the class
\begin{equation}
\label{E:homoga}
a^\chi:= a_0 + a_1 \chi + a_2 \chi^2 + \ldots \in A_0^{\C^*}(X) \/.
\end{equation}

Observe now that all the classes in the equation \eqref{E:iota} are $\C^*$-equivariant,
thus the formula extends to the equivariant context.
Further, by \cite[Proposition 2.7]{AMSS:shadows},
\begin{equation}
\label{E:c*iota}
\iota^*[C]_{\C^*}= (c(E) \cap s(C))^\chi \in A_0^{\C^*}(X) \/.
\end{equation}
(The proposition in {\em loc.cit.}~is stated in terms of the {\em shadow} of the cone $\overline{C}$;
by \cite[Lemma 2.1(a) and Lemma 2.2]{AMSS:shadows}, the shadow equals $c(E) \cap s(C)$.)
In \cref{sec:segre-csm}
we will be interested in the pull back $\iota^*[C]_{\C^*}$,
and we will use this formula to calculate it.

One of the fundamental properties of the Segre classes is their birational invariance, recalled next. 
Similar statements can be found in \cite[p.~10]{BBM:book} (without proof), and, for the 
\emph{normal cone} of a subvariety, in \cite[Proposition 4.2]{fulton:IT}.
For the convenience of the reader we include a proof.
\begin{lemma}
\label{lemma:segrepf}
Let $f:X' \to X$ be a proper morphism of irreducible non-singular varieties, $E \to X$ a vector bundle, and $C$ 
an irreducible subcone of $E$ over a closed subvariety $Y \subset X$. Let $Y' := f^{-1}(Y)$ and assume that 
$C'$ is a subcone of $f^*(E)$ over $Y'$ such that we have a commutative diagram
\[
\begin{tikzcd}
C' \arrow[r,"g"]\arrow[d] & C \arrow[d] \\ Y' \arrow[r,"f"] & Y
\end{tikzcd}
\]
where $g:C' \to C $ is proper and birational.
Then
\begin{equation}
\label{E:pf}
f_*(s(C')) =s(C) \in A_*(Y) \/.
\end{equation}
\end{lemma}

\begin{proof}
The morphism $f$ induces a morphism of vector bundles
$f^* (E)\oplus \one \to E \oplus \one$,
which in turn induces a morphism
$F: \bP(f^* (E) \oplus \one) \to \bP(E \oplus \one)$ between projective completions.
Let $G:\overline{C'}= \bP(C' \oplus \one) \to \bP(C \oplus \one)$ be the restriction. 
There is a commutative diagram 
\[
\begin{tikzcd}
\P(C' \oplus \one) \arrow[r,"G"]\arrow[d,"q'"] & \P(C \oplus \one) \arrow[d,"q"] \\ Y' \arrow[r,"f"] & Y
\end{tikzcd}
\]
The birationality of $g$ implies that
\begin{equation}
\label{E:pushC}
G_*[\overline{C'}] =[\overline{C}] \in A_*(\P(E \oplus \one)) \/.
\end{equation}
Now, since $C'$ is a subcone of $f^*(E)_{|Y'}$, we have
\begin{equation}
\label{E:O-1}
G^*(\cO_{C}(-1)) = \cO_{C'}(-1) \/,
\end{equation}
as both sides are the restriction of $\cO_{f^*(E)}(-1)=f^* \cO_{E}(-1)$ to $Y'$.
Following the definition of the Segre class (\cref{def:segre-class}), we have,
\[
\begin{split}
f_*(s(C'))&  = f_* q'_*\left(\frac{[\overline{C'}]}{c(\cO_{C'}(-1))}\right)
=q_* G_*\Bigl(\frac{[\overline{C}]}{G^*c(\cO_{C}(-1))}\Bigr) \\
& = q_* \Bigl(\frac{[\overline{C}]}{c(\cO_{C}(-1))}\Bigr) = s(C) \/.
\end{split}
\]
Here the third equality uses the projection formula and equations \eqref{E:pushC} and \eqref{E:O-1}.
This finishes the proof.
\end{proof}

All results extend naturally to the case where
$X$ is a variety with an action of a torus $T$,
$C$ is a $T \times \C^*$-invariant cone,
and the map $C \to X$ is $T \times \C^*$ equivariant (the $\C^*$ acting trivially on $X$).
For instance, in equation \eqref{E:c*iota}, the class $\iota^*[C]_{T \times \C^*}$ belongs to
$A_0^{T \times \C^*}(X)$, the $T \times \C^*$ equivariant Chow group.

\subsection{Conormal spaces}
The cones most important in this note are the conormal spaces of subvarieties,
whose definition we recall next.
Let $X$ be a smooth, irreducible, complex algebraic variety, and let $Y \subset X$ be a
closed irreducible subscheme.
Let $Y^{reg}$ be any any smooth dense set of $Y$.
The {\em conormal space} $T^*_YX$ is the closure of the conormal bundle $T^*_{Y^{reg}}X$
inside the cotangent bundle $T^*X$.
This is a cone in the sense of the previous section,
and also a closed subvariety of dimension $\dim X$,
contained in the restriction $T^*X_{|Y}$.
In particular, it is stable under the $\C^*$-dilation on the fibres of $T^*X$, and also under any 
group $G$ leaving $Y$ and $X$ invariant.
If one regards $T^*X$ as a symplectic manifold, then the conormal space is an irreducible conic Lagrangian cycle.
In fact, any irreducible, conic Lagrangian cycle is the conormal cone of some subvariety; 
see~\cite[Thm. E.6]{HTT} (where it is attributed to Kashiwara) and also \cite[\S 1]{kennedy:specialization}.

\section{Mather classes and CSM classes}\label{sec:segre-csm}
A question with a long and distinguished history is to define analogues of the total Chern class for singular varieties.
The Mather classes and the Chern-Schwartz-MacPherson (CSM) classes, considered in this note, are among these classes.
We recall their definition next.

\subsection{Mather classes}
Let $X$ be a smooth complex algebraic variety, and let $Y \subset X$ be a closed irreducible subvariety.
The {\em Mather class} of $Y$ is a non-homogenous homology class $\cma(Y) \in A_*(Y)$ with the property that if $Y$ is smooth then $\cma(Y) = c(T Y )\cap [Y]$.
Its original definition involves the Nash blowup of $Y$, but for the purpose of this note we use a variant of a result of Sabbah \cite{sabbah:quelques} (see also \cite{ginzburg:characteristic,PP:hypersurface}) relating the Mather class to the the class of the conormal space of $Y$ in $X$.
This variant appeared in \cite[Corollary 4.5 and Corollary 3.4]{AMSS:shadows}.
Further, we  work in the equivariant context, using the equivariant Mather class defined by Ohmoto \cite{ohmoto:eqcsm}; the corresponding class is denoted by $\cMaT(Y) \in A_*^T(X)$.

\begin{theorem}
[cf. \cite{AMSS:shadows}]
\label{thm:segre-mather}
Let $Y\subseteq X$ be a $T$-stable closed irreducible subvariety of the smooth variety $X$ and assume that $\C^*$ acts by dilation on the cotangent bundle $T^*X$ with character $\hbar^{-1}$.
Then the homogenization (cf.~ \eqref{E:homoga} above) of the $T$-equivariant Chern-Mather class satisfies
\begin{equation}
\label{E:cmath}
\cMaT(Y)^\hbar = (-1)^{\dim Y}\iota^*[T^*_YX]_{T \times \C^*}\/,
\end{equation}
as classes in $A_0^{T \times \C^*}(X)$.
\end{theorem}
It will be convenient to work with a dehomogenized variant of this equation.
Recall that by equation \eqref{E:c*iota} above,
\begin{align*}
\iota^*[T^*_Y(X)]_{T \times \C^*}= (c^T(T^*X) \cap s^T(T^*_Y X))^{-\hbar},
\end{align*}
since the $\C^*$ action is induced by $\hbar^{-1}$.
By equation \eqref{E:cmath} this implies that
\[
(-1)^{\dim Y}\cMaT(Y)^{-\hbar}= (c(T^*X) \cap s^T(T^*_Y(X)))^{\hbar}\/.
\]
After dehomogenizing, i.e. setting $\hbar= 1$, we obtain the expression
\begin{equation}
\label{E:cmatv}
\cMaTv(Y)= c^T(T^*X) \cap s^T(T^*_Y(X)) \/,
\end{equation}
where $\cMaTv(Y):= ((-1)^{\dim Y}\cMaT(Y)^{-\hbar})|_{\hbar = 1}$.

In other words, the class $\cMaTv(Y)$ is obtained from $\cMaT(Y)$
by changing signs of each homogeneous component according to its cohomological degree.
This is called the {\em dual Chern-Mather class};
it appears naturally when relating Chern-Mather classes to characteristic cycles on the cotangent bundle; cf.~\cite{sabbah:quelques}.

\subsection{Chern-Schwartz-MacPherson classes}
Let $X$ be any complex algebraic variety endowed with a Whitney stratification $\{ S_i \}$ of smooth constructible subsets.
Such a stratification always exists;
see \cite[Thm.~2.2]{verdier:stratifications} for the algebraic context,
and \cite[Thm.~19.2]{whitney:tangents} for the analytic context.
(Later, $X$ will be a Schubert variety with the stratification given by its Schubert cells.)

Denote by $\mathcal{F}(X)$ the group of constructible functions of $X$.
Its elements are finite sums of the form $\sum a_i \one_{W_i}$,
where $a_i \in\mathbb Z$, the $W_i \subset X$ are constructible subsets,
and $\one_{W_i}$ is the indicator function, which equals $1$ for points on $W_i$ and $0$ otherwise.
There are push-forward and pull-back operations defined as follows.
If $f: Z \to X$ is a {\em proper} morphism, then $f_*(\one_W)(x) = \chi(f^{-1}(x) \cap W)$,
where $\chi$ denotes the topological Euler characteristic;
one extends this further by linearity.
For any morphism $f:Z \to X$,
the pull back $f^*: \mathcal{F}(X) \to \mathcal{F}(Z)$ is defined by $f^*(\varphi)(z) = \varphi(f(z))$,
for $\varphi \in \mathcal{F}(X)$.

Proving a conjecture of Grothendieck and Deligne,
MacPherson \cite{macpherson:chern} defined a transformation $c_*:\mathcal{F}(X) \to H_*(X)$
which satisfies $c_*(\one_X) = c(T(X)) \cap [X]$ if $X$ is smooth,
and is functorial with respect to proper morphisms $f:Z \to X$.
This means that there is a commutative diagram
\[
\begin{tikzcd}
\mathcal{F}(Z) \arrow[r,"c_*"] \arrow[d,"f_*"] & H_*(Z) \arrow[d,"f_*"] \\
\mathcal{F}(X) \arrow[r,"c_*"]                 & H_*(X)
\end{tikzcd}
\]
If $W \subset X$ is a constructible subset,
the class $\csm(W):= c_*(\one_W) \in H_*(X)$ is called the {\em Chern-Schwartz-MacPherson (CSM)} class of $W$.
One may regard the CSM classes as an analogue of 
the total Chern class of the tangent bundle of $X$ in the case $X$ is singular.

MacPherson's definition of the transformation $c_*$ uses Mather classes,
and a constructible function $\Eu_X$ on $X$, called the {\em local Euler obstruction}.
The original definition of the local Euler obstruction in \cite{macpherson:chern}
uses transcedental methods (the analytic topology).
Later, Gonzalez-Sprinberg and Verdier \cite{gonzalez-sprinberg}, found an algebraic definition,
thus extending MacPherson's transformation to one with values in the Chow group $A_*(X)$.
More recently, Ohmoto \cite{ohmoto:eqcsm} generalized this to the equivariant context.
We recall the following properties of $\Eu_X$ - see \cite{macpherson:chern,gonzalez-sprinberg,brasselet.schwartz}:

\begin{lemma}
\label{lemma:euler}
\begin{enumerate}[label=(\alph*)]
\item
The local Euler obstruction $\Eu_X$ is constant along the strata of any Whitney stratification.
\item
$\Eu_X(x) =1$ if $X$ is nonsingular at $x$.
\item
If $X= X_1 \times X_2$ as varieties, then $\Eu_{X_1 \times X_2}(x_1,x_2) = \Eu_{X_1}(x_1) \cdot \Eu_{X_2}(x_2)$.
\end{enumerate}
\end{lemma}
\begin{proof}
Property (a) follows from \cite[Lemma 2]{macpherson:chern}, see also \cite[Prop. 10.1 and Corollaire 10.2]{brasselet.schwartz}.
The properties (b) and (c) are explicitly stated in \cite[\S 3]{macpherson:chern} and \cite[\S 4.2]{gonzalez-sprinberg}.
\end{proof}

We could not find a precise reference for the Proposition below,
although we believe it to be known to experts.

\begin{proposition}
\label{prop:eulerpb}
Let $f:Z \to X$ be a smooth morphism of nonsingular complex varieties,
and let $Y \subset X$ be a closed subvariety.
Then for any $z \in f^{-1}(Y)$, we have $\Eu_{f^{-1}(Y)}(z) = \Eu_Y(f(z))$,
i.e. as constructible functions $f^* \Eu_Y = \Eu_{f^{-1}(Y)}$.
\end{proposition}
\begin{proof}
Let $z \in f^{-1}(Y)$ and let $d:=\dim Z - \dim X$.
Since $f:Z \to X$ is a smooth morphism of relative dimension $d$,
\cite[Lemma 29.34.20]{StacksProj} implies that there exists an open affine neighborhood $U$ of $z$,
an open affine neighborhood $V$ of $f(z)$ such that $f(U) \subset V$,
and a commutative diagram
\[
\begin{tikzcd}
Z  \arrow[d,"f"] & U \arrow[l] \arrow[d,swap,"f_{|U}"] \arrow[r,"\eta"] & \mathbb{A}^d_V \arrow[dl] \\ X & V \arrow[l] &
\end{tikzcd}
\]
where $\eta$ is {\'e}tale.
From the definition, the local Euler obstruction only depends on the local behavior in the {\em analytic} topology, and this implies that $\Eu_{f^{-1}(Y)}(z) = \Eu_{f^{-1}(Y \cap V)}(z)$.
From the diagram above it follows that $\eta$ provides a local isomorphism in analytic topology between $f^{-1}(V \cap Y)$ and $(V \cap Y) \times \mathbb{A}^d$.
Then the claim follows from the product formula in (c) and again by using the local behavior and part (b).
\end{proof}

By definition, the Euler obstruction can be written as $\Eu_X = \sum e_i \one_{S_i}$, where $S_i \subset S$ is constructible and $e_i=\Eu_X(x_i)$ for any $x_i \in S_i$.
Then the Mather class and the MacPherson transformation are related by
\begin{equation}
\label{E:MatherEu}
\cMa(X) = c_*(\Eu_X) \/.
\end{equation}
In terms of CSM classes, this can be expressed as
\begin{equation}
\label{E:MaCSM}
\cMa(X) = \sum_i e_i \csm(S_i) \/.
\end{equation}
For $\varphi$ a constructible function on $X$, let
\[ s(\varphi) = \frac{c_*(\varphi)}{c(TX)} \]
denote the {\em Segre-MacPherson (SM)} class.
The following Verdier-Riemann-Roch (VRR) type theorem was proved by Yokura \cite{yokura:verdier}.

\begin{theorem}
\label{thm:VRR}
Assume that $f:Z \to X$ is a smooth morphism of complex algebraic varieties.
Then for any constructible function $\varphi \in \mathcal{F}(X)$, $f^* s(\varphi) = s(f^*(\varphi))$.
Equivalently, if $T_f$ denotes the relative tangent bundle of $f$, then
\[
c_*(f^*(\varphi)) = c(T_f) \cap  f^*(c_*(\varphi)) \/,
\]
as elements in $A_*(X)$.
\end{theorem}
\Cref{prop:eulerpb} implies that if $f:Z \to X$ is a smooth morphism, then 
$f^*(\Eu_Y)= \Eu_{f^{-1}(Y)}$.
If one takes $\varphi = \Eu_Y$, this implies that in terms of Mather classes
\begin{equation}
\label{E:eulerpb}
\cMa(f^{-1}(Y)) = c(T_f) \cap f^*(\cMa(Y)) \quad \in A_*(Z) \/.
\end{equation}

In \cref{sec:matherpb} below we will give another proof of this result,
in the case of Mather classes of Schubert varieties,
in the case when $f$ the projection between two (generalized) flag manifolds.

As usual, the results from this section can be extended to the case when all varieties have a torus $T$ action,
and all morphisms are $T$-equivariant.
The local Euler obstruction is the same, but one uses an equivariant Whitney stratification, and Ohmoto's equivariant version of MacPherson's transformation $c_*$ \cite{ohmoto:eqcsm}; see also \cite{AMSS:shadows}.

\section{Preliminaries on flag manifolds and cominuscule spaces}\label{sec:preliminaries}

\subsection{Preliminaries}
References for this section are \cite{kumar:book} and \cite{brion:flagv}. Let $G$ be a complex semisimple Lie group and fix a pair of opposite Borel subgroups, $B$ and $ B^-$ in $G$.
The opposite Borel subgroups determine a maximal torus $T:= B \cap B^-$, and a root system $R\subset\operatorname{Hom}(T,\mathbb{C}^*)$.

Let $R= R^+ \sqcup R^-$ be the decomposition into positive and negative roots, and let $\Delta \subset R^+$ be the set of simple roots.
We have a partial order on $R^+$, given by $\alpha < \beta$ if $\beta - \alpha $ is a non-negative combination of positive roots.

The Weyl group $W:=N_G(T)/T$ associated to $(G,T)$ is a Coxeter group generated by the simple reflections $s_i:=s_{\alpha_i}$, for $\alpha_i \in \Delta$.
Denote by $\ell:W \to \mathbb{N}$ the length function and by $w_0$ the longest element.

Any subset $S \subset \Delta$ determines a standard parabolic subgroup $P\supset B$.
We denote by $R_P^+$ the subset of $R^+$ consisting of roots whose support is contained in $S$.
The Weyl group $W_P$ of $P$ is generated by the simple reflections $s_i$, for $\alpha_i \in S$.
Denote by $w_P$ the longest element in $W_P$, and let $W^P$ be the set of minimal length representatives for the cosets in $W/W_P$.
If $w \in W$, the coset $wW_P$ has a unique minimal length representative $w^P \in W^P$ and as usual we set $\ell(wW_P) := \ell(w^P)$.

Let $G/P$ be the generalized flag manifold; this is a projective manifold of dimension $\ell(w_0W_P)$.
If $w \in W^P$ is a minimal length representative, the $B$-orbit $X_w^{P,\circ}= B w P/P$, and the $B^-$-orbit $(X^{P})^{w,\circ}= B^- w P/P$, are opposite Schubert cells for $w$.
With this definition, we have isomorphisms, $X_w^{P,\circ}\simeq \C^{\ell(w)}$ and $X^{w,P,\circ}\simeq \C^{\dim G/P - \ell(w)}$.
The \emph{Schubert varieties} $X_w^P$ and $X^{w,P}$ are the closures of the Schubert cells $X_w^{P,\circ}$ and $X^{w,P,\circ}$ respectively.

Every $P$-representation $V$ determines a $G$-equivariant vector bundle, $G \times^P V \to G/P$.
The points of $G\times^PV$ are equivalence classes $[g,v]$,
for pairs $(g,v) \in G \times V$ such that $(g,v) \simeq (gp^{-1},pv)$,
and the $G$-action on $G\times^PV$ is given by left multiplication, $g.[g',v] := [gg',v]$.
The main examples considered in this note are the following.

If $P=B$ is a Borel subgroup, we will take $V:= \C_\lambda$, the one dimensional $B$-module of character $\lambda$.
The resulting line bundle is $\mathcal{L}_\lambda:=G \times^B \C_\lambda$.

Let $\mathfrak{p}$ and $\mathfrak{g}$ be the Lie algebras of $P$ and $G$ respectively.
The group $P$ acts on $\mathfrak{p}$ and $\mathfrak{g}$ via the adjoint action.
Setting $V:=\mathfrak{g}/\mathfrak{p}$ in the construction above, we obtain the tangent bundle $T(G/P)=G \times^P \mathfrak{g}/\mathfrak{p}$.

Let $U_P$ be the unipotent radical of $P$.
The subspace $\mathfrak{u}_P := \operatorname{Lie}(U_P)$ is stable under the adjoint action of $P$ on $\mathfrak{p}$.
Following \cite{MR0263830}, we have a $P$-module isomorphism,
\[
\mathfrak g/\mathfrak p = \mathfrak{u}_P^*=\bigoplus_{\alpha \ge \alpha_P}\mathfrak{g}_{-\alpha}\/,
\]
where $\mathfrak{g}_\alpha$ is the one dimensional root subspace of $\mathfrak{g}$ corresponding to the root $\alpha$.

\subsection{Cominuscule spaces}
We recall next the basic definitions on cominuscule spaces; see e.g. \cite{BCMP:qkchev}.
A maximal parabolic subgroup of $G$ is determined upto conjugacy by removing a simple root $\alpha_P$ from $\Delta$.
We say that a maximal parabolic subgroup is \emph{cominuscule} if the corresponding simple root $\alpha_P$ appears with coefficient $1$ in the highest root from $R^+$;
the associated flag manifold $G/P$ is called a \emph{cominuscule space}.

The Dynkin diagrams, along with the possible choices of cominuscule nodes, are listed in \cref{rootPoset}.
The cominuscule spaces are classified as follows:
\begin{itemize}
\item The Grassmann manifolds $\Gr(k,n)$ if $G$ is of type $A$.
\item The Lagrangian Grassmannian $\mathrm{LG}(n,2n)$ in type C - this parametrizes vector subspaces of dimension $n$ in $\C^{2n}$, isotropic with respect to a non-degenerate skew-symmetric quadratic form.
\item The maximal orthogonal Grassmannian $\mathrm{OG}(n,2n)$ in type D - this the connected variety parametrizing vector subspaces of dimension $n$ in $\C^{2n}$, isotropic with respect to a non-degenerate quadratic form.
\item Quadrics in type B and D.
\item There are also two cominuscule spaces in the exceptional types $E_6$ and $E_7$ called respectively the Cayley plane and the Freudenthal variety.
\end{itemize}
The $n^{th}$ node of $B_n$ corresponds to the maximal orthogonal Grassmannian $\mathrm{OG}(n,2n+1)$.
This is also cominuscule, but when regarded as a homogeneous space under the type D group.
Indeed, this space parametrizes maximal subspaces which are isotropic with respect 
to a symmetric non-degenerate form in $\C^{2n+1}$.
From this description it follows that 
$\mathrm{OG}(n,2n+1)$ is isomorphic to either of the connected components of
$\mathrm{OG}(n+1,2n+2)$, and this isomorphism preserves Schubert classes;
see~\cite[p.~68]{fulton.pragacz:schubert} or e.g.~
\cite[\S 3.4]{IMN:factorial} for further details.
Similarly, the space corresponding to the first node in type C is the projective space $\bP^{2n-1}$.
This is a cominuscule space when regarded as a homogeneous space of type A.  

\begin{table}
\centering
\captionsetup[subfloat]{labelformat=empty}
\begin{tabular}{  c  c  }
\begin{minipage}{.34\textwidth}

{ \def\k{3} \def\d{4} 
\subfloat[][$\Gr(\k,\pgfmathparse{int(\k+\d)}\pgfmathresult)$]
{ 
\begin{tikzpicture}[scale=.7]
\draw (0.4,-0.3) rectangle (\d+1.1,\k+.8);
\foreach \x in {1,...,\d}
{
\foreach \y in {1,...,\k}
{
  \draw[thick,fill=white!70] (\x,\y) circle (.25cm);
  \draw (\x,\y) node{$\scriptstyle{\pgfmathparse{int(\x+\y-1)}\pgfmathresult}$};
}
}
\foreach \x in {2,...,\d}
{
  \foreach \y in {1,...,\k}
  {
    \draw[xshift=7,thick] (\x-1, \y) -- +(.5,0);
  }
}

\foreach \x in {1,...,\d}
{
  \foreach \y in {2,...,\k}
  {
    \draw[yshift=7,thick] (\x, \y-1) -- +(0,.5);
  }
}

\foreach \y in {2,\k}
{
  \draw (\d+.5 , \y-0.5) node{$\scriptstyle{+\alpha_{\pgfmathparse{int(\y-1)}\pgfmathresult}}$};
}

\foreach \x in {2,...,\d}
{
  \draw (\x-.5 , \k+.5) node{$\scriptstyle{+\alpha_{\pgfmathparse{int(\x+\d-2)}\pgfmathresult}}$};
}

\def\xl{.8}
\def\yl{.3}
\def\xw{.7}
\def\n{7}

\foreach \x in {1,...,\k} 
{
 \draw (\xw*\x-\xw+\xl , \yl ) circle (.1cm);
 \draw (\xw*\x-1*\xw+\xl , \yl-.5*\xw ) node {$\scriptscriptstyle{\pgfmathparse{int(\x)}\pgfmathresult}$};
 \draw[xshift=3] (\xw*\x-\xw+\xl , \yl) --+(\xw-.2,0);
}

\foreach \x in {2,...,\d} 
{
 \draw (\xw*\x-2*\xw+\k*\xw+\xl , \yl ) circle (.1cm);
 \draw (\xw*\x-2*\xw+\k*\xw+\xl , \yl-.5*\xw ) node {$\scriptscriptstyle{\pgfmathparse{int(\k+\x-1)}\pgfmathresult}$};
 \draw[xshift=3] (\xw*\x+\k*\xw-3*\xw+\xl , \yl) --+(\xw-.2,0);
}
\draw[thick,fill=black!80] (\xl+\k*\xw-\xw,\yl) circle (.1cm);


\end{tikzpicture}
 }\hfill }

{\def\n{4} 
\subfloat[][$\LG(\n,\pgfmathparse{int(2*\n)}\pgfmathresult)$]
{ 
\begin{tikzpicture}[scale=.7]
\draw (-\n-.5,\n+.75) rectangle (.1,.6);
\foreach \x in {1,...,\n}
{
  \foreach \y in {\x,...,\n}
  {
    \draw[thick,fill=white!70] (-\x,\y) circle (.25cm);
    \draw (-\x,\y) node{$\scriptstyle{\pgfmathparse{int(\n+\x-\y)}\pgfmathresult}$};
  }
}
\foreach \x in {2,...,\n}
{
  \draw (-\x+0.5,\n+0.5) node{$\scriptstyle{+\alpha_{\pgfmathparse{int(\x-1)}\pgfmathresult}}$};
  \draw (-.5,\x-0.5) node{$\scriptstyle{+\alpha_{\pgfmathparse{int(\n-\x+1)}\pgfmathresult}}$};
  \foreach \y in {\x,...,\n}
  {
    \draw[xshift=7,thick] (-\x, \y) -- +(.5,0);
    \draw[yshift=-7,thick] (1-\x, \y) -- +(0,-.5);
  }
}

\def\xl{-4.2}
\def\yl{1.3}
\def\xw{.7}

\foreach \x in {1,...,\n} 
{
 \draw (\xw*\x-\xw+\xl , \yl ) circle (.1cm);
 \draw (\xw*\x-1*\xw+\xl , \yl-.5*\xw ) node {$\scriptscriptstyle{\pgfmathparse{int(\x)}\pgfmathresult}$};
}
\foreach \x in {3,...,\n} \draw[xshift=3] (\xw*\x-3*\xw+\xl , \yl) --+(\xw-.2,0);

\draw[xshift=3,yshift=1] (\xw*\n-2*\xw+\xl , \yl) --+(\xw-.2,0);
\draw[xshift=3,yshift=-1] (\xw*\n-2*\xw+\xl , \yl) --+(\xw-.2,0);
\draw (\xw*\n-1.55*\xw+\xl , \yl) --+(.2*\xw,.2*\xw);
\draw (\xw*\n-1.55*\xw+\xl , \yl) --+(.2*\xw,-.2*\xw);
\draw[fill=black!80] (\xl+\n*\xw-\xw , \yl ) circle (.1cm);

\end{tikzpicture}
 }  \hfill }

\subfloat[][$\OG(6,12)$]
{ \begin{tikzpicture}[scale=.65]
\draw (-4.6,4.9) rectangle (1.1,-.4);
\foreach \n in {4}
{
  \foreach \x in {0,...,\n}
  {
    \foreach \y in {\x,...,\n}
      \draw[thick,fill=white!70] (-\x,\y) circle (.25cm);
  }
  \foreach \x in {1,...,\n}
  {
    \foreach \y in {\x,...,\n}
    {
      \draw[xshift=7,thick] (-\x, \y) -- +(.5,0);
      \draw[yshift=-7,thick] (1-\x, \y) -- +(0,-.5);
    }
  }
  \draw (-3.5 , 4.5) node{$\scriptstyle{+\alpha_4}$};
  \draw (-2.5 , 4.5) node{$\scriptstyle{+\alpha_3}$};
  \draw (-1.5 , 4.5) node{$\scriptstyle{+\alpha_2}$};
  \draw (-.5  , 4.5) node{$\scriptstyle{+\alpha_1}$};
  \draw (.5   , 3.5) node{$\scriptstyle{+\alpha_5}$};
  \draw (.5   , 2.5) node{$\scriptstyle{+\alpha_4}$};
  \draw (.5   , 1.5) node{$\scriptstyle{+\alpha_3}$};
  \draw (.5   , 0.5) node{$\scriptstyle{+\alpha_2}$};
}
\draw[thick,fill=gray!70] (-4,4) node {$\scriptstyle{6}$};
\draw[thick,fill=gray!70] (-3,4) node {$\scriptstyle{4}$};
\draw[thick,fill=gray!70] (-2,4) node {$\scriptstyle{3}$};
\draw[thick,fill=gray!70] (-1,4) node {$\scriptstyle{2}$};
\draw[thick,fill=gray!70] (0,4) node {$\scriptstyle{1}$};
\draw[thick,fill=gray!70] (-3,3) node {$\scriptstyle{5}$};
\draw[thick,fill=red!90] (-2,3) node {$\scriptstyle{4}$};
\draw[thick,fill=white!70] (-1,3) node {$\scriptstyle{3}$};
\draw[thick,fill=white!70] (0,3) node {$\scriptstyle{2}$};
\draw[thick,fill=gray!70] (-2,2) node {$\scriptstyle{6}$};
\draw[thick,fill=gray!70] (-1,2) node {$\scriptstyle{4}$};
\draw[thick,fill=white!70] (0,2) node {$\scriptstyle{3}$};
\draw[thick,fill=white!70] (-1,1) node {$\scriptstyle{5}$};
\draw[thick,fill=white!70] (0,1) node {$\scriptstyle{4}$};
\draw[thick,fill=white!70] (0,0) node {$\scriptstyle{6}$};

\def\n{6}
\def\xl{-4.3}
\def\yl{0.6}
\def\xw{.7}

\foreach \x in {4,...,\n} 
{
 \draw (\xw*\x-3*\xw+\xl , \yl ) circle (.1cm);
 \draw[xshift=3] (\xw*\x-4*\xw+\xl , \yl) --+(\xw-.2,0);
 \draw (\xw*\x-3*\xw+\xl , \yl-.5*\xw ) node {$\scriptscriptstyle{\pgfmathparse{int(\x-2)}\pgfmathresult}$};
}
\draw (\xl , \yl ) circle (.1cm);
\draw (\xl , \yl-.5*\xw ) node {$\scriptscriptstyle1$};
\draw (\xl+\n*\xw-2*\xw , \yl+\xw ) node {$\scriptscriptstyle{\pgfmathparse{int(\n-1)}\pgfmathresult}$};
\draw (\xl+\n*\xw-2*\xw , \yl-\xw ) node {$\scriptscriptstyle{\pgfmathparse{int(\n)}\pgfmathresult}$};
\draw[yshift=2,xshift=2] (\xl+\n*\xw-3*\xw , \yl) --+(\xw*.53,\xw*.53);
\draw[yshift=-2,xshift=2] (\xl+\n*\xw-3*\xw , \yl) --+(\xw*.53,-\xw*.53);
\draw (\xl+\n*\xw-2.27*\xw , \yl+.73*\xw) circle (.1cm);
\draw[fill=black!80] (\xl+\n*\xw-2.27*\xw , \yl-.73*\xw) circle (.1cm);

\end{tikzpicture} }  \hfill

{ \def\n{5}
\subfloat[][Even Quadric $Q^{\pgfmathparse{int(2*\n-2)}\pgfmathresult}$]
{ 
\begin{tikzpicture}[scale=.65]
\draw (1-\n,1.8) rectangle (\n-2.6,-1.35);
\foreach \x in {2,...,\n}
{
  \draw[thick,fill=white!70] (\x-3,0) circle (.25cm);
  \draw[thick,fill=white!70] (2-\x,1) circle (.25cm);
  \draw (2-\x,1) node{$\scriptstyle{\pgfmathparse{int(\n-\x+1)}\pgfmathresult}$};
}
\foreach \x in {4,...,\n} \draw (2.5-\x , 1.5) node{$\scriptstyle{+\alpha_{\pgfmathparse{int(\n-\x+2)}\pgfmathresult}}$};
\foreach \x in {3,...,\n}
{
  \draw[xshift=7,thick] (\x-4, 0) -- +(.5,0);
  \draw[xshift=7,thick] (2-\x, 1) -- +(.5,0);
  \draw (\x-3.5 , 1.5) node{$\scriptstyle{+\alpha_{\pgfmathparse{int(\n-\x+2)}\pgfmathresult}}$};
  \draw (\x-3,0) node{$\scriptstyle{{\pgfmathparse{int(\n-\x+1)}\pgfmathresult}}$};
}
\draw[yshift=7,thick] (-1, 0) -- +(0,.5);
\draw[yshift=7,thick] (0 , 0) -- +(0,.5);
\draw (.5 , 0.5) node{$\scriptstyle{+\alpha_\n}$};
\draw (-1 , 0) node{$\scriptstyle{\n}$};

\def\xl{-3.7}
\def\yl{-.4}
\def\xw{.7}

\foreach \x in {4,...,\n} 
{
 \draw (\xw*\x-3*\xw+\xl , \yl ) circle (.1cm);
 \draw[xshift=3] (\xw*\x-4*\xw+\xl , \yl) --+(\xw-.2,0);
 \draw (\xw*\x-3*\xw+\xl , \yl-.5*\xw ) node {$\scriptscriptstyle{\pgfmathparse{int(\x-2)}\pgfmathresult}$};
}
\draw[fill=black!80] (\xl , \yl ) circle (.1cm);
\draw (\xl , \yl-.5*\xw ) node {$\scriptscriptstyle1$};
\draw (\xl+\n*\xw-2*\xw , \yl+\xw ) node {$\scriptscriptstyle{\pgfmathparse{int(\n-1)}\pgfmathresult}$};
\draw (\xl+\n*\xw-2*\xw , \yl-\xw ) node {$\scriptscriptstyle{\pgfmathparse{int(\n)}\pgfmathresult}$};
\draw[yshift=2.2,xshift=2] (\xl+\n*\xw-3*\xw , \yl) --+(\xw*.53,\xw*.53);
\draw[yshift=-2.2,xshift=2] (\xl+\n*\xw-3*\xw , \yl) --+(\xw*.53,-\xw*.53);
\draw (\xl+\n*\xw-2.27*\xw , \yl+.73*\xw) circle (.1cm);
\draw (\xl+\n*\xw-2.27*\xw , \yl-.73*\xw) circle (.1cm);

\end{tikzpicture}
 }\hfill }

\end{minipage}
&
\begin{minipage}{.55\textwidth}

{ \def\n{5}
\subfloat[][Odd Quadric $Q^{\pgfmathparse{int(2*\n-1)}\pgfmathresult}$]
{ 
\begin{tikzpicture}[scale=.7]
\draw (.5-\n,2) rectangle (\n-.5,-0.5);
\foreach \x in {1,...,\n}
{
  \draw[thick,fill=white!70] (\x-1,1) circle (.25cm);
  \draw[thick,fill=white!70] (1-\x,1) circle (.25cm);
  \draw (1-\x,1) node{$\scriptstyle{\pgfmathparse{int(\n-\x+1)}\pgfmathresult}$};
  \draw (\x-1,1) node{$\scriptstyle{\pgfmathparse{int(\n-\x+1)}\pgfmathresult}$};
}
\foreach \x in {2,...,\n}
{
  \draw[xshift=7,thick] (\x-2, 1) -- +(.5,0);
  \draw[xshift=7,thick] (1-\x, 1) -- +(.5,0);
  \draw (-\x+1.5 , 1.5) node{$\scriptstyle{+\alpha_{\pgfmathparse{int(\n-\x+2)}\pgfmathresult}}$};
  \draw (\x-1.5 , 1.5) node{$\scriptstyle{+\alpha_{\pgfmathparse{int(\n-\x+2)}\pgfmathresult}}$};
}

\def\xl{-1.45}
\def\yl{0.15}
\def\xw{.7}

\foreach \x in {1,...,\n} 
{
 \draw (\xw*\x-\xw+\xl , \yl ) circle (.1cm);
 \draw (\xw*\x-1*\xw+\xl , \yl-.5*\xw ) node {$\scriptscriptstyle{\pgfmathparse{int(\x)}\pgfmathresult}$};
}
\foreach \x in {3,...,\n} \draw[xshift=3] (\xw*\x-3*\xw+\xl , \yl) --+(\xw-.2,0);

\draw[xshift=3,yshift=1] (\xw*\n-2*\xw+\xl , \yl) --+(\xw-.2,0);
\draw[xshift=3,yshift=-1] (\xw*\n-2*\xw+\xl , \yl) --+(\xw-.2,0);
\draw (\xw*\n-1.4*\xw+\xl , \yl) --+(-.2*\xw,.2*\xw);
\draw (\xw*\n-1.4*\xw+\xl , \yl) --+(-.2*\xw,-.2*\xw);
\draw[fill=black!80] (\xl , \yl ) circle (.1cm);

\end{tikzpicture}
 }\hfill 
}

\subfloat[][$E_6/P_6$]
{ \begin{tikzpicture}[scale=.7]

\draw (-6,5) rectangle (3,.5);
\foreach \x in {-2,...,1}
  \draw[xshift=7,thick] (\x, 1) -- +(.5,0);
    
\foreach \x in {1,...,4}
{
  \foreach \y in {\x,...,4}
  {
    \draw[thick,fill=white!20] (-\x,\y) circle (.25cm);
  }
}
\foreach \x in {2,...,4}
{
  \foreach \y in {\x,...,4}
  {
    \draw[xshift=7,thick] (-\x, \y) -- +(.5,0);
    \draw[yshift=-7,thick] (1-\x, \y) -- +(0,-.5);
  }
}
\draw (-3.5 , 4.5) node{$\scriptstyle{+\alpha_4}$};
\draw (-4.5 , 4.5) node{$\scriptstyle{+\alpha_5}$};
\draw (-2.5 , 4.5) node{$\scriptstyle{+\alpha_3}$};
\draw (-1.5 , 4.5) node{$\scriptstyle{+\alpha_1}$};
\draw (-.5  , 4.5) node{$\scriptstyle{+\alpha_3}$};
\draw (0.5  , 4.5) node{$\scriptstyle{+\alpha_4}$};
\draw (1.5  , 4.5) node{$\scriptstyle{+\alpha_2}$};
\draw (2.4  , 3.6) node{$\scriptstyle{+\alpha_2}$};
\draw (2.4  , 2.6) node{$\scriptstyle{+\alpha_4}$};
\draw (2.4  , 1.6) node{$\scriptstyle{+\alpha_5}$};

\draw[thick,fill=green] (-5,4) circle (.25cm);
\draw[thick,fill=green] (-4,4) circle (.25cm);
\draw[thick,fill=green] (-3,4) circle (.25cm);
\draw[thick,fill=green] (-2,4) circle (.25cm);
\draw[thick,fill=green] (-1,4) circle (.25cm);
\draw[thick,fill=green] (-3,3) circle (.25cm);
\draw[thick,fill=green] (-2,3) circle (.25cm);
\draw[thick,fill=green] (-2,2) circle (.25cm);
\draw[thick,fill=white!20] (0,2) circle (.25cm);
\draw[thick,fill=white!20] (0,1) circle (.25cm);
\draw[thick,fill=white!20] (1,1) circle (.25cm);
\draw[thick,fill=white!20] (2,1) circle (.25cm);
\draw[thick,fill=white!20] (-2,1) circle (.25cm);

\draw[thick,fill=green] (-5,4) node {$\scriptstyle6$};
\draw[thick,fill=green] (-4,4) node {$\scriptstyle5$};
\draw[thick,fill=green] (-3,4) node {$\scriptstyle4$};
\draw[thick,fill=green] (-2,4) node {$\scriptstyle3$};
\draw[thick,fill=green] (-1,4) node {$\scriptstyle1$};
\draw[thick,fill=green] (-3,3) node {$\scriptstyle2$};
\draw[thick,fill=red!90] (-2,3) node {$\scriptstyle4$};
\draw[thick,fill=white!20] (-1,3) node {$\scriptstyle3$};
\draw[thick,fill=green] (-2,2) node {$\scriptstyle5$};
\draw[thick,fill=green] (-1,2) node {$\scriptstyle4$};
\draw[thick,fill=white!20] (0,2) node {$\scriptstyle2$};
\draw[thick,fill=white!20] (-2,1) node {$\scriptstyle6$};
\draw[thick,fill=white!20] (-1,1) node {$\scriptstyle5$};
\draw[thick,fill=white!20] (0,1) node {$\scriptstyle4$};
\draw[thick,fill=white!20] (1,1) node {$\scriptstyle3$};
\draw[thick,fill=white!20] (2,1) node {$\scriptstyle1$};

\draw[xshift=7  , thick] (-5 , 4) -- +(.5 , 0);
\draw[xshift=7  , thick] (-1 , 2) -- +(.5 , 0);
\draw[yshift=-7 , thick] (0  , 2) -- +(0  , -.5);
\draw[yshift=-7 , thick] (-2 , 2) -- +(0  , -.5);
\draw[xshift=7  , thick] (-1 , 1) -- +(.5 , 0);

\def\xl{-5.5}
\def\yl{1.5}
\def\xw{.7}

\foreach \x in {0,...,3}
{
 \draw (\xw*\x+\xl , \yl ) circle (.1cm);
 \draw[xshift=3] (\xw*\x+\xl , \yl) --+(\xw-.2,0);
 \draw (\xw*\x+\xw+\xl , \yl-.5*\xw ) node {$\scriptscriptstyle{\pgfmathparse{int(\x+3)}\pgfmathresult}$};
}
\draw (\xl , \yl-.5*\xw ) node {$\scriptscriptstyle1$};
\draw (\xl+2.5*\xw , \yl+\xw ) node {$\scriptscriptstyle2$};
\draw[yshift=3] (\xl+2*\xw , \yl) --+(0,\xw-.2);
\draw[thick,fill=black!80] (\xl+4*\xw , \yl) circle (.1cm);
\draw (\xl+2*\xw , \yl+\xw) circle (.1cm);

\end{tikzpicture} }  \hfill

\subfloat[][$E_7/P_7$]
{ \begin{tikzpicture}[scale=.68]

\draw (-6.35,5) rectangle (2.95,-4.5);
\foreach \x in {-2,...,1}
  \draw[xshift=7,thick] (\x, 1) -- +(.5,0);
    
\foreach \x in {1,...,4}
{
  \foreach \y in {\x,...,4}
  {
    \draw[thick,fill=white!20] (-\x,\y) circle (.25cm);
  }
}
\foreach \x in {2,...,4}
{
  \foreach \y in {\x,...,4}
  {
    \draw[xshift=7,thick] (-\x, \y) -- +(.5,0);
    \draw[yshift=-7,thick] (1-\x, \y) -- +(0,-.5);
  }
}
\draw (-5.5 , 4.5) node{$\scriptstyle{+\alpha_6}$};
\draw (-4.5 , 4.5) node{$\scriptstyle{+\alpha_5}$};
\draw (-3.5 , 4.5) node{$\scriptstyle{+\alpha_4}$};
\draw (-2.5 , 4.5) node{$\scriptstyle{+\alpha_3}$};
\draw (-1.5 , 4.5) node{$\scriptstyle{+\alpha_1}$};
\draw (-.5  , 4.5) node{$\scriptstyle{+\alpha_3}$};
\draw (0.5  , 4.5) node{$\scriptstyle{+\alpha_4}$};
\draw (1.5  , 4.5) node{$\scriptstyle{+\alpha_2}$};
\draw (2.5  , 3.5) node{$\scriptstyle{+\alpha_2}$};
\draw (2.5  , 2.5) node{$\scriptstyle{+\alpha_4}$};
\draw (2.5  , 1.5) node{$\scriptstyle{+\alpha_5}$};
\draw (2.5  , 0.5) node{$\scriptstyle{+\alpha_6}$};
\draw (2.5  ,-0.5) node{$\scriptstyle{+\alpha_5}$};
\draw (2.5  ,-1.5) node{$\scriptstyle{+\alpha_4}$};
\draw (2.5  ,-2.5) node{$\scriptstyle{+\alpha_3}$};
\draw (2.5  ,-3.5) node{$\scriptstyle{+\alpha_1}$};

\foreach \y in {0,...,3}
{
  \draw[yshift=-7 , thick] (2  , -\y) -- +(0 , -.5);
  \draw[xshift=7  , thick] (\y-2 , 0) -- +(.5  , 0);
  \draw[thick , fill=white!20] (2  ,\y-4) circle (.25cm);
}

\foreach \x in {-6,...,-1}
  \draw[thick , fill=white!20] (\x , 4) circle (.25cm);

\foreach \y in {-2,...,2}
{
  \draw[yshift=-7 , thick] (\y , 1) -- +(0 , -.5);
  \draw[thick , fill=white!20] (\y , 0) circle (.25cm);
}
\draw[thick , fill=white!20] (-3 , 3) circle (.25cm);
\draw[thick , fill=white!20] (-2 , 3) circle (.25cm);
\draw[thick , fill=white!20] (-2 , 2) circle (.25cm);
\draw[thick , fill=white!20] (0  , 2) circle (.25cm);
\draw[thick , fill=white!20] (0  , 1) circle (.25cm);
\draw[thick , fill=white!20] (1  , 1) circle (.25cm);
\draw[thick , fill=white!20] (2  , 1) circle (.25cm);
\draw[thick , fill=white!20] (-2 , 1) circle (.25cm);
\draw[thick , fill=white!20] (1  ,-1) circle (.25cm);

\draw[thick , fill=white!20] (-6 , 4) node {$\scriptstyle7$};
\draw[thick , fill=white!20] (-5 , 4) node {$\scriptstyle6$};
\draw[thick , fill=white!20] (-4 , 4) node {$\scriptstyle5$};
\draw[thick , fill=white!20] (-3 , 4) node {$\scriptstyle4$};
\draw[thick , fill=white!20] (-2 , 4) node {$\scriptstyle3$};
\draw[thick , fill=white!20] (-1 , 4) node {$\scriptstyle1$};
\draw[thick , fill=white!20] (-3 , 3) node {$\scriptstyle2$};
\draw[thick , fill=white!20] (-2 , 3) node {$\scriptstyle4$};
\draw[thick , fill=white!20] (-1 , 3) node {$\scriptstyle3$};
\draw[thick , fill=white!20] (-2 , 2) node {$\scriptstyle5$};
\draw[thick , fill=white!20] (-1 , 2) node {$\scriptstyle4$};
\draw[thick , fill=white!20] (0  , 2) node {$\scriptstyle2$};
\draw[thick , fill=white!20] (-2 , 1) node {$\scriptstyle6$};
\draw[thick , fill=white!20] (-1 , 1) node {$\scriptstyle5$};
\draw[thick , fill=white!20] (0  , 1) node {$\scriptstyle4$};
\draw[thick , fill=white!20] (1  , 1) node {$\scriptstyle3$};
\draw[thick , fill=white!20] (2  , 1) node {$\scriptstyle1$};
\draw[thick , fill=white!20] (-2 , 0) node {$\scriptstyle7$};
\draw[thick , fill=white!20] (-1 , 0) node {$\scriptstyle6$};
\draw[thick , fill=white!20] (0  , 0) node {$\scriptstyle5$};
\draw[thick , fill=white!20] (1  , 0) node {$\scriptstyle4$};
\draw[thick , fill=white!20] (2  , 0) node {$\scriptstyle3$};
\draw[thick , fill=white!20] (1  ,-1) node {$\scriptstyle2$};
\draw[thick , fill=white!20] (2  ,-1) node {$\scriptstyle4$};
\draw[thick , fill=white!20] (2  ,-2) node {$\scriptstyle5$};
\draw[thick , fill=white!20] (2  ,-3) node {$\scriptstyle6$};
\draw[thick , fill=white!20] (2  ,-4) node {$\scriptstyle7$};

\draw[xshift=7  , thick] (-6 , 4) -- +(.5  , 0);
\draw[xshift=7  , thick] (-5 , 4) -- +(.5  , 0);
\draw[xshift=7  , thick] (-1 , 2) -- +(.5  , 0);
\draw[xshift=7  , thick] (-1 , 1) -- +(.5  , 0);
\draw[xshift=7  , thick] (1  , -1) -- +(.5 , 0);

\draw[yshift=-7 , thick] (0  , 2) -- +(0 , -.5);
\draw[yshift=-7 , thick] (-2 , 2) -- +(0 , -.5);
\draw[yshift=-7 , thick] (1  , 0) -- +(0 , -.5);

\def\xl{-5}
\def\yl{-2.5}
\def\xw{.7}

\foreach \x in {0,...,4} 
{
 \draw (\xw*\x+\xl , \yl ) circle (.1cm);
 \draw[xshift=3] (\xw*\x+\xl , \yl) --+(\xw-.2,0);
 \draw (\xw*\x+\xw+\xl , \yl-.5*\xw ) node {$\scriptscriptstyle{\pgfmathparse{int(\x+3)}\pgfmathresult}$};
}
\draw (\xl , \yl-.5*\xw ) node {$\scriptscriptstyle1$};
\draw (\xl+2.5*\xw , \yl+\xw ) node {$\scriptscriptstyle2$};
\draw[yshift=3] (\xl+2*\xw , \yl) --+(0,\xw-.2);
\draw[thick,fill=black!80] (\xl+5*\xw , \yl) circle (.1cm);
\draw (\xl+2*\xw , \yl+\xw) circle (.1cm);

\end{tikzpicture} }  \hfill

\end{minipage}

\end{tabular}
\caption{
Hasse diagrams for the lattices \rp. 
The top left node is the minimal root $\alpha_P$, and the bottom right node is the highest root.
See also \cref{sec:tables}.
}
\label{rootPoset}
\end{table}

\section{A resolution of singularities for the conormal spaces of cominuscule Schubert varieties}
\label{sec:resolution}
By a cominuscule Schubert variety, we simply mean a Schubert subvariety of a cominuscule space.
We recall here the construction from \cite{singh:conormalII} of a resolution of singularities of the conormal space of a cominuscule Schubert variety.
This is the main tool required to calculate (equivariant) Mather classes of these Schubert varieties.

Fix a minimal length representative $w \in W^P$, and let $\underline w=(s_{i_1},\cdots,s_{i_k})$ be a reduced word for $w$.
Consider the vector subspace
\begin{align*}
\u_w=\bigoplus\limits_{\substack{\alpha\geq\alpha_P\\ w(\alpha)>0}}\mathfrak g_\alpha \subset \u_P \/.
\end{align*}
A key fact proved in \cite[Lemma 2.1]{singh:conormalII} is that $\u_w$ is a $B$-submodule of $\u_P$ under the adjoint action.
To this data we can associate the Bott-Samelson variety $B_{\underline{w}}:= P_{i_1}\times^B P_{i_2}\times^B \ldots \times^B P_{i_k}/B$, a $B$-variety under left multiplication, and the vector bundle
\[
\mathcal{E}_{\underline{w}}:= P_{i_1}\times^B P_{i_2}\times^B \ldots \times^B P_{i_k}\times^B \mathfrak{u}_w \to B_{\underline{w}}\/,
\]
see, for example, \cite[Ch.2]{brion.kumar:frobenius}.
We have a commutative diagram
\begin{equation}
\label{E:diag}
\begin{tikzcd}
\mathcal E_{\underline{w}}\arrow[r,"\theta_{\underline{w}}'"] \arrow[d] & \overline{BwB}\times^B \mathfrak{u}_w \arrow[r, "\pi'"] \arrow[d] & T^*_{X_w^P}(G/P) \arrow[r,hook] \arrow[d] & T^*(G/P)\arrow[d] \\
B_{\underline{w}}\arrow[r,"\theta_{\underline{w}}"]                     & X_w^B \arrow[r,"\pi"]                                             & X_w^P \arrow[r,hook]                      & G/P,
\end{tikzcd}
\end{equation}
where the maps are defined as follows:
\begin{itemize}
\item the morphism $\theta_{\underline{w}}$ is birational, and it is the usual projection from the Bott-Samelson desingularization of the Schubert variety $X_w^B$;
the middle vertical map is the vector bundle projection;
\item the left square is a fibre square, inducing the morphisms $\theta_{\underline{w}}'$ and the left vertical map (also a vector bundle projection);
\item the rightmost vertical morphism is the usual projection from the cotangent bundle of $G/P$;
\item $\pi$ is the restriction to $X_w^B$ of the usual projection $G/B \to G/P$;
\item $\pi'$ is obtained by the composition of the morphisms
\begin{align*}
&&\overline{BwB}\times^B \mathfrak u_w \hookrightarrow \overline{BwB}\times^B \mathfrak u_P \hookrightarrow G \times^B \mathfrak u_P \to G \times^P \mathfrak u_P = T^*(G/P) \/.
\end{align*}
Each of these morphisms is proper, hence the composition is also proper.
It was proved in \cite[Thm.  A]{singh:conormalII} that $\pi'$ is birational.
\end{itemize}

Combining everything proves that the following holds \cite{singh:conormalII}:
\begin{theorem}
\label{thm:rescon}
The composition,
\[
\pi' \circ \theta_{\underline{w}}':\mathcal{E}_w \to T^*_{X_w^P}(G/P),
\]
is proper and birational,
hence a resolution of singularities of the conormal space $T^*_{X_w^P}(G/P)$.
\end{theorem}

\section{Mather classes of Schubert varieties}\label{sec:mather-schubert}
In this section we prove the formula calculating the Mather classes of Schubert varieties in cominuscule spaces, and we illustrate the calculation with several examples.

\subsection{The formula for Mather classes}
Let $\mathcal{U}_w:= \overline{BwB}\times^B \mathfrak{u}_w \to X_w^B$ denote the restriction of the homogeneous bundle $G \times^B \mathfrak{u}_w$ to the Schubert variety $X_w^B$.
There is an exact sequence of homogeneous vector bundles on $X_w^B$ given by
\begin{equation}
\label{E:exact}
\begin{tikzcd}
0 \arrow[r] & \mathcal{U}_w \arrow[r] & \pi^* T^*(G/P)_{|X_w^B}\arrow[r] & \mathcal{T}_w^* \arrow[r] & 0
\end{tikzcd}
\end{equation}
where $\mathcal{T}_w^*:= \overline{BwB}\times^B (\mathfrak{u}_P/\mathfrak{u}_w)$.
Observe that $\mathcal{T}_w^*$ restricted to the open Schubert cell in $X_w^B$ is the cotangent bundle of this cell (explaining the choice of notation).
\Cref{thm:rescon} and the considerations from diagram \eqref{E:diag} imply that the diagram
\begin{equation}
\label{E:conres}
\begin{tikzcd}
\mathcal{U}
_w \arrow[r, "\pi'"] \arrow[d] & T^*_{X_w^P}(G/P) \arrow[d] \\ X_w^B \arrow[r,"\pi"] & X_w^P
\end{tikzcd}
\end{equation}
satisfies the hypotheses in \cref{lemma:segrepf} with $X'=Y' = X_w^B$ and $\mathcal{U}_w$ a subbundle of $\pi^* T^*(G/P)_{|X_w^P}$.
This allows us to calculate the Mather class of the Schubert variety $X_w^P$.
Let $\mathcal{T}_w$ denote the dual of the bundle $\mathcal{T}_w^*$.

\begin{theorem}
\label{thm:mather}
Let $w \in W^P$ be a minimal length representative.
Then the equivariant Mather class of $X_w^P$ is
\[
\cMaT(X_w^P)=\pi_*(c^T(\mathcal{T}_w) \cap [X_w^B]) \/,
\]
where $ c^T(\mathcal{T}_w) = \prod_{\alpha \in I(w)}c^T(\mathcal{L}_{-\alpha})$.
\end{theorem}
\begin{proof}
It follows from the $T$-module decomposition,
\[
(\mathfrak{u}_P/\mathfrak{u}_w)^* = \bigoplus\limits_{\substack{\alpha\geq\alpha_P\\ w(\alpha)<0}}\mathfrak{g}_{-\alpha}= \bigoplus\limits_{\substack{\alpha\in I(w)}}\mathfrak{g}_{-\alpha}\/,
\]
that the total Chern class of the homogeneous vector bundle
$\mathcal{T}_w$ has the same localization at $T$-fixed points $e_v \in G/B$ as the Chern class of the vector bundle
\[
\bigoplus_{\alpha \in I(w)}G \times^B \C_{-\alpha}= \bigoplus_{\alpha \in I(w)}\mathcal{L}_{-\alpha}\/.
\]
The result for $c^T(\mathcal{T}_w)$ now follows from Whitney's formula.

To prove the formula for the Mather class,
first observe that the birationality property of Segre classes (see \cref{lemma:segrepf})
applied to the diagram \eqref{E:conres} yields
\begin{equation}
\label{E:SSM}
 s^T(T^*_{X_w^P}(G/P)) = \pi_*(s^T(\mathcal{U}_w) \cap [X_w^B]) = \pi_*( c^T(\mathcal{U}_w)^{-1}\cap [X_w^B]) \/.
\end{equation}
It follows from the \cref{E:cmatv} version of \cref{thm:segre-mather} that
\[
\begin{split}
\cMaTv(X_w^P)  & = c^T(T^*(G/P)) \cap s^T(T^*_{X_w^P}(G/P)) \\ & = c^T(T^*(G/P) \cap  \pi_*( c^T(\mathcal{U}_w)^{-1}\cap [X_w^B]))  \\ &  =  \pi_* (c^T(\mathcal{T}_w^*) \cap [X_w^B]) \/,
\end{split}
\]
where the last equality follows from the projection formula.
The proof ends by changing the signs in each homogeneous component;
this corresponds to taking the Chern classes of the dual bundle $\mathcal{T}_w$.
\end{proof}

Another algorithm to calculate Mather classes, in the case of Grassmannians, was found by B.  Jones \cite{jones:csm}.
He used Zelevinsky's small resolutions for Schubert varieties \cite{MR705051}, and 
equivariant localization, to calculate the 
{\em Kazhdan-Lusztig class} (KL) of a Schubert variety.
As explained in \cref{s:positivity} below (see also \cite[\S 6]{AMSS:shadows}) this coincides with the Mather class.
Sankaran and Vanchinathan \cite{sankaran.v:small-res} found Schubert varieties in the Lagrangian Grassmannian which do not admit small resolutions and Perrin \cite{perrin:small-res} characterized the minuscule Schubert varieties with this property.

\subsection{Examples}\label{ss:examples}
The previous theorem, combined with the (equivariant) Chevalley formula, gives an effective way to calculate the (equivariant) Mather class.
We recall the equivariant Chevalley formula, following \cite[Thm.~8.1]{buch.m:nbhds};
see also \cite[Cor.~11.3.17 and Thm.~11.1.7]{kumar:book}.

Let $\lambda$ be a weight, and $\mathcal{L}_{\lambda}= G \times^B \C_\lambda$ be the associated line bundle.
Then
\[
c_1^T(\mathcal{L}_\lambda) \cap [X_w^B] =  w(\lambda) [X_w^B] + \sum_\alpha \langle -\lambda, \alpha^\vee \rangle [X_{w s_\alpha}^B] \/,
\]
where the sum is over all positive roots $\alpha$ such that $\ell(w s_\alpha) = \ell(w) -1$.
\begin{footnote}
{The minus sign is explained by the fact that if $\omega_i$ is the $i^{th}$ fundamental weight, then non-equivariantly $c_1(\mathcal{L}_{-\omega_i})$ is the class of the Schubert divisor $X^{s_i,B}$, an effective class.}
\end{footnote}
Applying repeatedly this formula we can recursively calculate the expression
\[
c^T(\mathcal{T}_w) \cap [X_w^B] = \sum_{v \le w}a_{w,v} [X_v^B] \/,
\]
where $a_{w,v} \in A^*_T(pt)$.
For instance,
\[
a_{w,w} = \prod_{\alpha \in I(w)}(1-w(\alpha)) \/.
\]
Recall that
\begin{align*}
\pi_*[X_v^B] =
\begin{cases}
[X_v^P]&\text{if }v \in W^P,\\
0&\text{otherwise.}
\end{cases}
\end{align*}
It follows that the equivariant Mather class of $X_w$ equals
$\sum\limits_{v \le w; v \in W^P}a_{w,v} [X_v^P] \/$.

\begin{example}
We illustrate this calculation next.
Let $G= \mathrm{SL}_4$ and the simple roots $\alpha_1 = \varepsilon_1 - \varepsilon_2, \alpha_2= \varepsilon_2 - \varepsilon_3, \alpha_3 = \varepsilon_3- \varepsilon_4$ (notation as in \cite{bourbaki}).
The maximal parabolic $P$ associated to $\alpha_2$ gives the Grassmannian $G/P =\Gr(2,4)$.
One may identify the elements in $W^P$ by Young diagrams, which in turn may be used to read both the inversion set and the reduced word for $w$;
see e.g. \cite[\S 3]{BCMP:qkchev} and \cref{sec:tables} below.
For instance, in the table below, the green portion corresponds to the Schubert divisor $X_{s_1 s_3 s_2}^P \subset \Gr(2,4)$,
with inversion set $I(s_1 s_3 s_2) = \{\alpha_2,\varepsilon_2-\varepsilon_4,\varepsilon_1-\varepsilon_3 \}$.
The Schubert divisor is the smallest example of a singular Schubert variety in $\Gr(2,4)$:
it is a $3$ dimensional quadric singular at the point $1.P$.
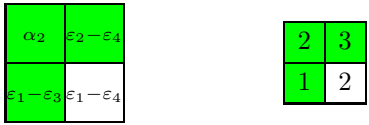
\begin{center}
\begin{tabular}{cc}
\ytableausetup{centertableaux,boxsize=2.2em}

\begin{ytableau}*(green) \scriptstyle \alpha_2 & *(green){ \scriptstyle \varepsilon_2-\varepsilon_4}\\
  *(green) {\scriptstyle \varepsilon_1-\varepsilon_3}&{\scriptstyle \varepsilon_1-\varepsilon_4} 
  \end{ytableau}\hskip2cm

\ytableausetup{centertableaux,boxsize=normal}
\begin{ytableau}*(green) 2 & *(green){3} \\*(green) {1}& {2} \end{ytableau}

\end{tabular}
\end{center}
The space $G/B$ is the complete flag manifold $\mathrm{Fl}(4)$.
The Chevalley formula in $A^*(\mathrm{Fl}(4))$ gives that
\begin{equation}
\label{E:132B}
\begin{split}
c(\mathcal{T}_w) \cap [X_{s_1 s_3 s_2}^B] = & [X_{s_1 s_3 s_2}^B]+3[X_{s_3 s_2}^B]+ 4[X_{s_3 s_1}^B]+3[X_{s_3}^B]+ 3 [X_{s_1 s_2}^B] \\ & +8[X_{s_2}^B]+3 [X_{s_1}^B]+ 6 [X_{id}^B] \/.
\end{split}
\end{equation}
Pushing forward to $\Gr(2,4)$, we obtain the Mather class:
\begin{equation}
\label{E:Ma312}
\begin{split} 
\cMa(X_{s_1 s_3 s_2}^P) = & [X_{s_1 s_3 s_2}^P]+3[X_{s_3 s_2}^P]+ 3 [X_{s_1 s_2}^P]+8[X_{s_2}^P]+ 6 [X_{id}^P]  \\
= & {\tableau{6}{{}& {}\\ {}&}} + 3~{\tableau{6}{{} & {}}}+ 3~ {\tableau{6}{{}\\ {}}}+ 8~ {\tableau{6}{{}}}+ 6 ~{\emptyset}
\/.
\end{split}
\end{equation}
(We simply used $\lambda$ for the Schubert class indexed by $\lambda$.)
The equivariant calculation is more involved, and we present only the final answer.
\begin{equation}
\begin{split}
\cMaT (X_{s_1 s_3 s_2}^P) = & (1+\alpha_1)(1+\alpha_3)(1+\alpha_1+\alpha_2+\alpha_3) \tableau{6}{{}&{}\\ {}} \\ 
& + (1+ \al_3)(3+ \al_1 + 2 \al_2 + 2\al_3+3)\tableau{6}{{} & {}} \\
& + (1+ \al_1)(3+ 2\al_1+2 \al_2 + \al_3+3) \tableau{6}{{} \\ {}} \\
& + (8+2 \al_1 + 4 \al_2+ 2 \al_3) \tableau{6}{{}} + 6 ~\emptyset \/.
\end{split}
\end{equation}
\end{example}

\begin{example}
In our next example we consider $G= \mathrm{SL}_6$ and the Grassmannian $G/P = \Gr(3,6)$.
We consider the Schubert variety indexed by the green portion in the diagrams below.
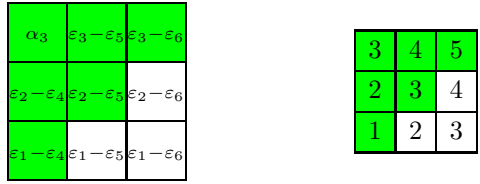
\begin{center}
\begin{tabular}{cc}
\ytableausetup{centertableaux,boxsize=2.2em}
\begin{ytableau}
*(green) \scriptstyle \alpha_3                      & *(green){\scriptstyle \varepsilon_3-\varepsilon_5}  & *(green){\scriptstyle \varepsilon_3-\varepsilon_6}\\
*(green) {\scriptstyle \varepsilon_2-\varepsilon_4} & *(green) {\scriptstyle \varepsilon_2-\varepsilon_5} & { \scriptstyle \varepsilon_2-\varepsilon_6}\\
*(green) \scriptstyle \varepsilon_1-\varepsilon_4   & {\scriptstyle \varepsilon_1-\varepsilon_5}          & {\scriptstyle \varepsilon_1-\varepsilon_6}
\end{ytableau} \hskip2cm
\ytableausetup{centertableaux,boxsize=normal}
\begin{ytableau}
*(green) 3   & *(green){4}  & *(green) {5} \\
*(green) {2} & *(green) {3} & {4} \\
*(green) 1   & 2            & 3
\end{ytableau}
\end{tabular}
\end{center}
Then $w = s_1 s_3 s_2 s_5 s_4 s_3$ and the inversion set consists of the roots in the green boxes.
The Mather class equals
\[
\begin{split}
\cMa(\tableau{6}{{} & {} & {} \\ {} & {} \\ {}}) & =  \tableau{6}{{} & {} & {} \\ {} & {} \\ {}} + 4 \tableau{6}{{} & {} & {} \\ {} & {}} + 4 \tableau{6}{{} & {} & {} \\ {} \\ {}} + 4 \tableau{6}{{} & {} \\ {} & {} \\ {}} + 15 \tableau{6}{{} & {} & {} \\ {}} + 15 \tableau{6}{{} & {} \\ {} & {}} + 15 \tableau{6}{{} & {} \\ {} \\ {}} \\ & + 17 \tableau{6}{{} & {} & {}} + 52 \tableau{6}{{} & {} & \\ {}} + 17 \tableau{6}{{}\\ {} & \\ {}} + 54 \tableau{6}{{} & {}} + 54 \tableau{6}{{} \\ {}} + 60 \tableau{6}{{}} + 24 ~\emptyset \/.
\end{split}
\]
This is in accordance \cite[Table 2]{jones:csm}.
In fact, we were able to recover all calculations from {\em loc.~cit.} - see \cref{sec:tables} below.
\end{example}

\begin{example}
\label{Ex:LG431}
We now consider the symplectic group $\mathrm{Sp}_8$ and the Lagrangian Grassmannian $\LG(4,8)$.
The elements in $W^P$ are indexed by strict partitions included in the $(4,3,2,1)$ staircase.
Take $\lambda = (4,3,1)$ corresponding to the green boxes in the diagram below.
As usual, the notation for the roots follows \cite{bourbaki};
we refer the reader to \cite[\S3]{BCMP:qkchev} and \cref{sec:tables} below for further combinatorial details. 
\begin{center}
\begin{tabular}{cc}
\ytableausetup{centertableaux,boxsize=2.2em}
\begin{ytableau}*(green) \scriptstyle \alpha_4 & *(green){\scriptstyle \varepsilon_3+\varepsilon_4} 
&  *(green){\scriptstyle \varepsilon_2+\varepsilon_4} 
& *(green) {\scriptstyle \varepsilon_1+\varepsilon_4}
\\ \none &  *(green) {\scriptstyle 2\varepsilon_3} & *(green) { \scriptstyle \varepsilon_2+\varepsilon_3} 
& *(green) \scriptstyle \varepsilon_1+\varepsilon_3 
\\ \none & \none & *(green) {\scriptstyle 2 \varepsilon_2} & {\scriptstyle \varepsilon_1+\varepsilon_2} 
\\ \none & \none & \none & {\scriptstyle 2 \varepsilon_1} \end{ytableau} \hskip2cm
\ytableausetup{centertableaux,boxsize=normal}
\begin{ytableau}*(green) 4 & *(green){3} &  *(green){2} & *(green) {1}
\\ \none &  *(green) {4} & *(green) {3} & *(green) 2 \\ \none & \none & *(green) {4} & {3} 
\\ \none & \none & \none & {4} \end{ytableau} 
\end{tabular}
\end{center}
In this case the reduced word is $w=s_4 s_2 s_3 s_4 s_1 s_2 s_3 s_4$,
and the inversion set consists of the entries in the green boxes of the first diagram.
The Mather class is given by 
\[
\begin{split}
\cMa\bigl(\tableau{6}{{} & {} & {} & {} \\ & {} & {} & {} \\ & & {}}\bigr)  & = \tableau{6}{{} & {} & {} & {} \\ & {} & {} & {} \\ & & {}} + 4 \tableau{6}{{} & {} & {} & {} \\ & {} & {} & {}} + 7 \tableau{6}{{} & {} & {} & {} \\ & {} & {} \\ & & {}} + 27 \tableau{6}{{} & {} & {} & {} \\ & {} & {}} + 25 \tableau{6}{{} & {} & {} \\ & {} & {} \\ & & {}} + 60 \tableau{6}{{} & {} & {} & {} \\ & {}} + 92 \tableau{6}{{} & {} & {} \\ & {} & {}} \\ & + 45 \tableau{6}{{} & {} & {} & {}} + 241 \tableau{6}{{} & {} & {} \\ & {}} + 183 \tableau{6}{{} & {} & {}} + 269 \tableau{6}{{} & {}\\ & {}} + 246 \tableau{6}{{} & {}} + 132 \tableau{6}{{}} + 24 ~\emptyset \/.
\end{split}
\]
\end{example}

\begin{remark}
\label{rm:gg}
We observe that the bundle $\mathcal{T}_w$ is {\em not} globally generated, even when restricted to $X_w^B$;
if it were, then multiplying its total non-equivariant Chern class by any Schubert class would have to be effective.
To see that this is not the case, consider the situation above when $G/P = \Gr(2,4)$, $w=s_1 s_3 s_2 \in W^P$, and take $u=s_3$.
Then $X_u^B \subset X_w^B$ and $c(\mathcal{T}_w) \cap [X_u^B] = [X_{s_3}^B] - [X_{id}^B]$.
Despite this, examples suggest that the Mather classes are effective;
see \cref{sec:positivity} below for more about this. 
\end{remark}

More examples are included in \cref{sec:tables} below.

\section{A cohomological formula for the local Euler obstruction}
\label{sec:eulerobs}
The goal of this section is to prove \cref{thm:eulerobs},
which gives a formula for the local Euler obstruction function for cominuscule Schubert varieties.
Different formulae in the classical Lie types were also obtained by Boe and Fu \cite{boe.fu}.

Because the CSM class of the Schubert cell $\csm(X_u^{P,\circ})$ is a deformation of the fundamental class $[X_u^P]$ of the Schubert variety,
the CSM classes $\csm(X_u^{P,\circ})$ ($u \in W^P$) form a basis of the Chow group $A_*(G/P)$.
It was proved in \cite{AMSS:shadows} that the (Poincar{\'e}) dual basis of
the CSM basis of Schubert cells with respect to the intersection pairing
is the family of the {\em Segre-Schwartz-MacPherson} (SSM) classes $\ssm((X^P)^{v,\circ})$ ($v \in W^P$) of the opposite Schubert cells.

We recall next the relevant definitions.
The SSM class is defined by
\[
\ssm((X^P)^{v,\circ}) = \frac{\csm((X^P)^{v,\circ})}{c(T(G/P))}\/.
\]
The Chern class in the quotient is invertible because $c(T(G/P)) = 1+\kappa$, where $\kappa$ is a nilpotent element.
Then by \cite[Thm.9.4]{AMSS:shadows}, 
\begin{equation}
\label{E:duality}
\langle \csm(X_u^{P,\circ}), \ssm((X^P)^{v,\circ}) \rangle = \int_{G/P}\csm(X_u^{P,\circ}) \cdot \ssm((X^P)^{v,\circ}) = \delta_{u,v}\/.
\end{equation}
Combined with \cref{thm:mather}, this duality implies a cohomological formula for the Euler obstruction coefficients $e_{w,v}$ from equation \eqref{E:cMa};
we record this next.

\begin{theorem}
\label{thm:eulerobs}
Let $v,w \in W^P$ and assume that $v \le w$.
Then the local Euler obstruction coefficient $e_{w,v}$ is given by
\[
e_{w,v}= \sum \int_{G/B}\frac{c(\mathcal{T}_w) \cdot [X_w^B]}{c(T(G/B))}\cdot \csm((X^B)^{u,\circ})=  \sum \int_{G/B}c(\mathcal{T}_w) \cdot [X_w^B] \cdot \csmv((X^B)^{u,\circ}) \/,
\]
where the sum is over $u \in W; v \le u \le w$ such that $uW_P = vW_P$.
\end{theorem}

\begin{proof}
By the duality equation \eqref{E:duality}, and the projection formula, the Euler obstruction coefficients are given by
\[
e_{w,v}= \int_{G/P}\pi_*(c(\mathcal{T}_w) \cdot [X_w^B]) \cdot \ssm((X^P)^{v,\circ}) = \int_{G/B}c(\mathcal{T}_w) \cdot [X_w^B] \cdot \pi^*\ssm((X^P)^{v,\circ}) \/.\]
The Verdier-Riemann-Roch formula implies that
\[
\pi^*\ssm((X^P)^{v,\circ}) = \ssm(\pi^{-1}((X^P)^{v,\circ})) = \sum \ssm((X^B)^{u,\circ}) \/,
\]
where the sum is over all $u \ge v$ such that $uW_P = vW_P$.
Further, it is proved in \cite[Cor. 7.4]{AMSS:shadows} that for any $u \in W$,
\[
\ssm((X^B)^{u,\circ}) = \csmv ((X^B)^{u,\circ}) \/.
\]
Since the class $\csmv (X^B)^{u,\circ})$ is supported on the Schubert variety $X^{B,u}$, it follows that the product $[X_w^B] \cdot [X^{B,u}] = 0$ unless $u \le w$.
The claim follows by combining the three equations above.
\end{proof}

We note that an explicit calculation for the CSM classes was obtained in \cite{aluffi.mihalcea:eqcsm}.
Therefore the integrals from \cref{thm:eulerobs} can be explicitly computed in small examples,
using software such as the {\em Equivariant Schubert Calculator} by A. Buch.

\begin{example}
We continue with the example $G/P =\Gr(2,4)$ and $X_w^P = X_{s_1 s_3 s_2}^P$ the Schubert divisor.
Recall from equation \eqref{E:132B}:
\[
c(\mathcal{T}_w) \cap [X_{s_1 s_3 s_2}^B] =  [X_{s_1 s_3 s_2}^B]+3[X_{s_3 s_2}^B]+ 4[X_{s_3 s_1}^B]+3[X_{s_3}^B]+ 3 [X_{s_1 s_2}^B]+8[X_{s_2}^B]+3 [X_{s_1}^B]+ 6 [X_{id}^B] \/.
\]
If we take $v=s_3s_2$, then $u=v$ and we obtain that $\csmv(X^{B, s_3 s_2,\circ})$ equals
\begin{small}
\[
\begin{split}
 & [X^{B,s_3 s_2}]-[X^{B, s_3s_2 s_1}]-2[X^{B, s_3 s_1s_2}]+[X^{B, s_3 s_1 s_2 s_1}]+4[X^{B,s_2 s_3 s_1 s_2}] \\ & -2[X^{B,s_2 s_3 s_1 s_2 s_1}]-[X^{B,s_2 s_3 s_2}]+[X^{B, s_2s_3s_2s_1}]+3[X^{B,s_1s_2s_3s_2}]-2[X^{B,s_1s_2s_3s_2s_1}]\\ & +[X^{B,s_1s_2s_3s_1s_2s_1}]-3[X^{B,s_1s_2s_3s_1s_2}] \/.
\end{split}
\]
\end{small}
Since $\int_{G/B}[X_{v_1}^B] \cdot [X^{B,v_2}] = \delta_{v_1,v_2}$ we obtain that
\[
e_{s_1 s_3 s_2,s_3 s_2}= -2+3 = 1 \/.\] Of course, this was expected, as the Schubert divisor is only singular at the base point $1.P$.
\end{example}
\begin{example}
\label{ex:EulerLG24}
Consider the Lagrangian Grassmannian $\mathrm{LG}(2,4)$.
This is isomorphic to a $3$ dimensional quadric in $\bP^4$.
The set $W^P$ indexing the Schubert varieties is in bijection with the strict partitions in the $2 \times 2$ rectangle: $(0), (1), (2)$ and $(2,1)$.
The only singular Schubert variety is the divisor $X_{\tableau{5}{{}&{}}}^P$, with singularity at the point $X_{\emptyset}^P=1.P$.
One calculates that 
\[ e_{(2),(2)} = e_{(2),(1)} =1; \quad e_{(2),(0)} =0 \/. \]
This verifies examples from \cite[p.~456]{boe.fu}.
Using the isomorphism of $\mathrm{LG}(2,4)$ to the $3$-dimensional quadric, 
it also verifies one instance of \cite[Eq.(6.3.3)]{boe.fu}.
\end{example}

\begin{example}
\label{ex:EulerLG36}
We now consider the Lagrangian Grassmannian $\mathrm{LG}(3,6)$.
In this case the Schubert varieties are indexed by strict partitions in the $3 \times 3$ rectangle.
An interesting example is obtained by considering the divisor $X_{(3,2)}^P$.
In this case, the Euler obstructions are:
\begin{align*}
e_{(3,2),(3,2)} = e_{(3,2),(3,1)} = e_{(3,2),(2,1)} =  e_{(3,2),(0)}= 1 \\
e_{(3,2),(3)} = e_{(3,2),(2)} = e_{(3,2),(1)}=0 \/.
\end{align*}
Observe that the Euler obstruction at $X_{id}^P= 1.P$ is $1$, even though the variety is singular at that point. 
\end{example}

\section{Localization of conormal spaces}
\label{sec:localization}
The goal of this section is to use \cref{thm:mather} to obtain formulae for the localization of conormal spaces.

Denote by $\iota:G/P \hookrightarrow T^*(G/P)$ the zero section.
By equation \eqref{E:cmath},
\begin{equation}
\label{E:pbi}
\iota^*[T^*_{X_w^P}(G/P)]_{T \times \C^*} = (-1)^{\ell(w)}\cMaT(X_w^P)^\hbar = (-1)^{\ell(w)}\pi_*((c^T(\mathcal{T}_w) \cap [X_w^B])^\hbar) \/.
\end{equation}
This is a class in $A_0^{T \times \C^*}(G/P)$ and it is obtained by homogenizing the equivariant (homology) class $\cMaT(X_w^P)$.
We can use \eqref{E:pbi} to calculate the localization at the fixed points of the conormal space.
Fix $u \in W^P$ such that $u \le w$.
If we write
\[
\cMaT(X_w^T) = \sum a_v [X_v^P] \/,
\]
where $a_v \in A^*_T(pt)$, then localizing at $u \in W$ gives
\[
\cMaT(X_w^T)|_u = \sum a_v [X_v^P]|_u \/.
\]
The coefficients $a_v$ are not homogeneous, thus homogenizing each term $(a_v [X_v^P]|_u)$ amounts to multiplying each homogeneous component by the appropriate power of 
$\hbar \in A^1_{T \times \C^*}(pt)$ such that the total cohomological degree equals $\dim G/P$ (i.e. homological degree $0$.) 

A more explicit formula is obtained by analyzing the homogenization $(c^T(\mathcal{T}_w) \cap [X_w^B])^\hbar$.
Let $a_1, \ldots , a_{\ell(w)}$ be the $T$-equivariant Chern roots of $\mathcal{T}_w$.
Then
\[
c(\mathcal{T}_w) \cap [X_w^B] = \Bigl(\prod_{i=1}^{\ell(w)}(1 + a_i)\Bigr)\cap [X_w^B] = \sum_{i \ge 0} e_i (a_1, \ldots, a_{\ell(w)}) \cap [X_w^B] \/,
\]
where $e_i$ denote the elementary symmetric functions.
The term $e_i (a_1, \ldots, a_{\ell(w)}) \cap[X_w^B]$ belongs to the equivariant Chow group $A_{\ell(w) -i}^T(G/B)$, therefore its homogenization by $\hbar$ is
\begin{equation}
\label{E:chomog}
(c(\mathcal{T}_w) \cap [X_w^B])^\hbar = \sum_{i \ge 0} \hbar^{\ell(w) - i}e_i (a_1, \ldots, a_{\ell(w)}) \cap [X_w^B] = \Bigl(\prod_{i=1}^{\ell(w)}(\hbar + a_i)\Bigr) \cap [X_w^B] \/.
\end{equation}
One key observation is that the quantity $(-1)^{\ell(w)}\prod_{i=1}^{\ell(w)}(\hbar + a_i)$ has a geometric interpretation.
If one considers the $\C^*$-action on the cotangent bundle $T^*(G/B)$ with character $\hbar^{-1}$, then the elements $-\hbar -a_i$ are the $T \times \C^*$-equivariant Chern roots of $\mathcal{T}_w^*$.
Thus
\begin{equation}
\label{E:topTC}
(-1)^{\ell(w)}\prod_{i=1}^{\ell(w)}(\hbar + a_i) = c^{T \times \C^*}_{\ell(w)}(\mathcal{T}^*_w) \/.
\end{equation}
The fibre of $\mathcal{T}^*_w$ over the fixed point $e_v$ is
\[
v.\bigoplus_{\alpha \in I(w)}\mathfrak{g}_{\alpha}\otimes \C_{-\hbar}= \bigoplus_{\alpha \in I(w)}\mathfrak{g}_{v(\alpha)}\otimes \C_{-\hbar}\/,
\]
and from this we deduce the formula for the localization:
\[
(c^{T\times \C^*}_{\ell(w)}(\mathcal{T}_w^*))|_{v}= \prod_{\alpha \in I(w)}(-\hbar + v(\alpha)) \/.
\]
Combining with the equation \eqref{E:chomog}, this proves the following Lemma:
\begin{lemma}
\label{lemma:ctwloc}
Let $v \le w$.
Then the following holds in $A^{\dim G/B}_{T \times \C^*}(pt)$:
\[
(-1)^{\ell(w)}((c(\mathcal{T}_w) \cap [X_w^B])^\hbar)|_v =  \prod_{\alpha \in I(w)}(-\hbar + v(\alpha)) \cdot [X_w^B]|_v \/.
\]
\end{lemma}
We note that since $\C^*$ acts trivially on $G/B$, the $T \times \C^*$-localization $[X_w^B]|_v$ is the same as 
the $T$-equivariant localization.
A formula for the latter can be found in \cite[Thm. 11.1.7]{kumar:book}; 
see also \cite[App. D]{AJS:representations} or \cite{billey:kostant}.

The last step to calculate the localization of the class of the conormal space $T^*_{X_w^P}(G/P)$ is to relate the localization of the class from \cref{lemma:ctwloc} to the localization of its push forward.
For that, we need a generalization of the localization formula.
For $v \in W$, let $[e_v]$ denote the $T$-equivariant class of the fixed point $e_v$.
\begin{lemma}
\label{lemma:pf}
Let $u \in W^P$ and let $\kappa \in H^*_{T}(G/B)$ be any equivariant cohomology class.
Then
\[
\pi_*(\kappa)|_{uW_P}= \sum \frac{[e_{uW_P}]|_{uW_P}}{[e_v]|_v}\kappa|_v \/
\]
in an appropriate localization of $A^*_T(G/B)$, where the sum is over $v \in W$ such that $uW_P = v W_P$.
\end{lemma}
\begin{proof}
By injectivity of localization map,
the classes $[e_v]$ form a basis for the $T$-equivariant cohomology of $G/B$,
{\em localized} at the prime ideal generated by the equivariant parameters $A^*_{T}(pt)$.
Thus we can expand $\kappa = \sum c_v [e_v]$ where the sum is over $v \in W$.
Localizing both sides at $v \in W$, we obtain that $\kappa|_v = c_v [e_v]|_v$.
Pushing forward and localizing at $uW_P$ one obtains
\[
\pi_*(\kappa)|_{uW_P}= \sum_{v \in W}c_v \pi_*[e_v]|_{uW_P}= \sum_{v \in W}c_v [e_{vW_P}]|_{uW_P}= \sum_{v \in W}\frac{[e_{vW_P}]|_{uW_P}}{[e_v]|_v}\kappa|_v\/.
\]
Since $[e_{vW_P}]|_{uW_P}$ is nonzero only when $uW_P = v W_P$, the last sum is as in the statement of the lemma, and this finishes the proof.
\end{proof}

\begin{theorem}
\label{thm:conormalloc}
Let $u,w \in W^P$ such that $u \le w$, and let $\C^*$ act on $T^*(G/P)$ by character $\hbar^{-1}$.
Then the $T \times \C^*$-localization of the conormal space $T^*_{X_w^P}(G/P)$ is given by:
\[
[T^*_{X_w^P}(G/P)]|_{uW_P}=
\sum\limits_{\substack{v \le w \\ v W_P = u W_P}}\frac{\prod_{\alpha \in I(w)}(-\hbar + v(\alpha)) \cdot \prod_{\alpha \in R^+ \setminus R_P^+}u(-\alpha)}{\prod_{\alpha \in R^+}v(-\alpha)}[X_w^B]|_v \/.
\]
\end{theorem}
\begin{proof}
Let $\kappa = (-1)^{\ell(w)}(c^T(\mathcal{T}_w) \cap [X_w^B])^\hbar$, regarded as a cohomology class in $A^{*}_{T}(G/B)[\hbar]$.
By equation \eqref{E:pbi}, the left hand side equals to $\pi_*(\kappa)|_{uW_P}$. 
Since $\C^*$ acts trivially on $G/B$, we have $A^{*}_{T \times \C^*}(G/B)= A^{*}_{T}(G/B)[\hbar]$, 
and further, the projection $\pi_*$ is $A^*_{T \times \C^*}(pt)$-linear.
To finish the proof, apply \cref{lemma:ctwloc,lemma:pf}, using the fact that the $T \times \C^*$ equivariant Euler classes $[e_v]|_v$ and $[e_{uW_P}]|_{uW_P}$ coincide with the $T$-equivariant ones; further, the latter equal $[e_v]|_v = \prod_{\alpha \in R^+}v(-\alpha)$ and $[e_{uW_P}]|_{uW_P}= \prod_{\alpha \in R^+ \setminus R_P^+}u(-\alpha)$.
\end{proof}

\begin{example}
Let $u=w$.
The only $v$ satisfying the requirements is $v=w$.
Then
\[
\iota_w^*[T^*_{X_w^P}(G/P)]_{T \times \C^*}= \frac{\prod_{\alpha \in I(w)}(-\hbar +v(\alpha)) \cdot [X_w^B]|_w}{\prod_{\alpha \in R_P^+}v(-\alpha)}=
\prod_{\alpha \in I(w)}(-\hbar + w(\alpha)) \cdot [X_w^P]|_w \/,
\]
where the last equality follows from standard manipulations of (equivariant) Euler classes, for example by using \cref{lemma:pf}.
\end{example}

In \cite{conormal2017}, Lakshmibai and Singh identified certain conormal spaces as open subsets of affine Schubert varieties.
It would interesting to obtain localization formulae for the conormal spaces using localization for affine Schubert varieties.


\section{Mather classes of pull backs of Schubert varieties}
\label{sec:matherpb}
In this section we let $P$, $Q$ be two arbitrary parabolic subgroups such that $B \subset Q \subset P$;
we remove the cominuscule hypothesis.
Our goal is to give an alternative proof of the formula \eqref{E:eulerpb}
for the Mather classes of pull back Schubert varieties via the projection $\pi:G/Q \to G/P$.
Instead of analyzing the Euler obstruction, this proof focuses on the conormal spaces of Schubert varieties, and their relation to Mather classes.

Fix $w \in W^P$ and set $C:= T^*_{X_w^P}(G/P) \subset T^*(G/P)$, the conormal space of the Schubert variety $X_w^P$.
Consider the commutative diagram
\begin{equation}
\label{E:pbT}
\begin{tikzcd}
\rho_\pi \omega_\pi^{-1}(C) \arrow[d,hook'] & \omega_\pi^{-1}(C) \arrow[l,equal] \arrow[r] \arrow[d,hook'] & C \arrow[d,hook'] \\
T^*(G/Q) \arrow[dr] & G/Q \times_{G/P}T^*(G/P) \arrow[l,swap,hook',"\rho_\pi"] \arrow[d,"p"] \arrow[r, "\omega_{\pi}"] & T^*(G/P) \arrow[d] \\ & G/Q \arrow[bend left = 20]{ul}{\iota^Q}\arrow[bend left = 20]{u}{\iota}\arrow[r, "\pi"] & G/P \arrow[bend right = 20,swap]{u}{\iota^P}
\end{tikzcd}
\end{equation}
Here the downward vertical maps on the bottom right square, and the diagonal maps, are projections, the upward maps $\iota, \iota^Q, \iota^P$ are the zero sections, and the right squares are fibre squares.
The morphism $\rho_\pi$ is 
defined by $\rho_\pi(x,\xi) = (x, \xi \circ d\pi(x))$,
where $d\pi(x):T_x(G/P) \to T_{\pi(x)}(G/P)$ is the differential of $\pi$ at $x$.
Since $\pi$ is a smooth morphism, $\omega_\pi$ is smooth by base change,
and $\rho_\pi$ is a closed embedding; see e.g. \cite[p.65]{HTT}.

The following Lemma is well known;
see e.g.~\cite[Lemma 3]{kashiwara.tanisaki:characteristic} for a special case,
or \cite[Prop. 4.3.3]{dimca} for more general cases,
referring to \cite[pag. 231-232]{kashiwara.schapira:sheaves}.
For completeness we include a proof.
\begin{lemma}
\label{lemma:conpb}
Let $w \in W^P$.
Then
\[
\rho_{\pi}\omega_{\pi}^{-1}(T^*_{X_w^P}(G/P)) = T^*_{\pi^{-1}(X_w^P)}(G/Q)
\]
is the conormal space of the pull-back Schubert variety $\pi^{-1}(X_w^P)$.
\end{lemma}
\begin{proof}
The morphism $\pi$ is a locally trivial fibration with smooth connected fibre $F\simeq P/Q$.
Therefore $\omega_\pi$ is again a locally trivial fibration with fibre $F$, by base change.
Combining this with the fact that the conormal space $C:=T^*_{X_w^P}(G/P) \subset T^*(G/P)$ is an irreducible (conic) Lagrangian cycle, we obtain that $\omega_\pi^{-1}(C)$ is irreducible of dimension
\[
\dim \omega_{\pi}^{-1}(C) = \dim \pi^*(T^*(G/P)) - \dim G/P = \dim G/Q \/.
\]
It is easy to check that $\omega_{\pi}^{-1}(C)$ contains the conormal space of the pull back
$\pi^{-1}(X_w^{P,\circ})$ of the Schubert {\em cell}$X_w^{P,\circ}$.
The latter conormal space is an open set both in $\omega_{\pi}^{-1}(C)$ and in $T^*_{\pi^{-1}(X_w^P)}(G/Q)$.
The claim follows since $\omega_{\pi}^{-1}(C)$ is irreducible.
\end{proof}
Denote by $T_\pi$ the relative tangent bundle associated to the smooth morphism $\pi$.
This is a bundle on $G/Q$ of rank $\dim P/Q$,
determined by the following exact sequence of bundles on $G/Q$:
\[
\begin{tikzcd}
0 \arrow[r] & T_\pi \arrow[r] & T(G/Q) \arrow[r] & \pi^*(T(G/P)) \arrow[r] & 0
\end{tikzcd}
\]

Next we give another proof of the equation \eqref{E:eulerpb} for Schubert varieties, using conormal spaces.
\begin{theorem}
\label{thm:pbMa}
Let $w \in W^P$.
Then
\[
\cMaT(\pi^{-1}(X_w^P)) = c^{T}(T_\pi) \cap \pi^*(\cMaT(X_w^P)) \/.
\]
\end{theorem}
\begin{proof}
Denote by $n:=\dim Q/P$.
We will use the notation from diagram \eqref{E:pbT}.
Since $\rho_\pi$ is a closed embedding, and $\omega_\pi$ is smooth,
we may define the push-forward $(\rho_\pi)_*$ and the pull back $\omega_\pi^*$ in the appropriate Chow groups.
To prove the claim we utilize the formula \eqref{E:cmath}, and we calculate $(\iota^Q)^*[T^*_{\pi^{-1}(X_w^P)}(G/Q)]_{T \times \C^*}$:
\[
(\iota^Q)^*[T^*_{\pi^{-1}(X_w^P)}(G/Q)]_{T \times \C^*}= \iota^* \rho_\pi^*[T^*_{\pi^{-1}(X_w^P)}(G/Q)]_{T \times \C^*}= \iota^* \rho_\pi^* (\rho_\pi)_* \omega_\pi^*[C]_{T \times \C^*}\/,
\]
where the last equality follows from Lemma \ref{lemma:conpb}.
By the self-intersection formula \cite[Cor. 6.3]{fulton:IT}
\[
\rho_\pi^* (\rho_\pi)_* \omega_\pi^*[C]_{T \times \C^*}= c_n^{T \times \C^*}(N) \cap \omega_\pi^*[C]_{T \times \C^*}\/,
\]
where $N$ is the normal bundle of $G/Q \times_{G/P}\pi^*(T^*(G/P)) \subset T^*(G/Q)$.
As a $T \times \C^*$-equivariant bundle, the normal bundle is just the pull-back $N= p^* T^*_\pi$, where $p:G/Q \times_{G/P}T^*(G/P) \to G/Q$ is the projection.
Combining the last two equations and using that $\iota^*$ is a ring homomorphism, it follows that
\[
\begin{split}
(\iota^Q)^*[T^*_{\pi^{-1}(X_w^P)}(G/Q)]_{T \times \C^*}= & \iota^*(c_n^{T \times \C^*}(p^*T^*_\pi)) \cap \iota^* (\omega_\pi^*[C]_{T \times \C^*}) \\=  & c_n^{T \times \C^*}(T^*_\pi) \cap \iota^* (\omega_\pi^*[C]_{T \times \C^*}) \/;
\end{split}
\]
the last equality follows because $p \circ \iota = id_{G/Q}$.
Finally, we use that $\omega_\pi \circ \iota = \iota^P \circ \pi$ to obtain:
\[
\begin{split}
(\iota^Q)^*[T^*_{\pi^{-1}(X_w^P)}(G/Q)]_{T \times \C^*}= & c_n^{T \times \C^*}(T^*_\pi) \cap \iota^* (\omega_\pi^*[C]_{T \times \C^*}) \\  = & c_n^{T \times \C^*}(T^*_\pi) \cap \pi^* (\iota^P)^*[C] \\ = & (-1)^{\dim X_w^P}c_n^{T \times \C^*}(T^*_\pi) \cap \pi^*(\cMaT(X_w^P)^\hbar) \/.
\end{split}
\]
Since $\dim \pi^{-1}(X_w^P) = \dim X_w^P+ n$, this can be rewritten as
\begin{equation}
\label{E:int}
(-1)^n \cMaT(\pi^{-1}(X_w^P))^\hbar = c_n^{T \times \C^*}(T^*_\pi) \cap \pi^*(\cMaT(X_w^P)^\hbar) \/.
\end{equation}
Let $a_1, \ldots a_n$ denote the Chern roots of the $T$-equivariant bundle $T_\pi$.
Now, since $\C^*$ acts on the fibres of various vector bundles with character $\hbar^{-1}$, we deduce that the Chern roots of the $T \times \C^*$-equivariant bundle $T^*_\pi$ are $-\hbar-a_1, \ldots, -\hbar - a_n$.
This implies that
\[
c_n^{T \times \C^*}(T^*_\pi) = (-\hbar -a_1) \cdot \ldots \cdot (-\hbar - a_n) = (-1)^n c^T(T_\pi)^\hbar \/,
\]
where $c^T(T_\pi)^\hbar \in A^n_T(G/Q)$ denotes the homogenization of the total Chern class of $T_\pi$ by the character $\hbar$.
The theorem now follows from combining this with equation \eqref{E:int} and setting $\hbar = 1$.
\end{proof}
\begin{remark}
The same proof works when $\pi$ is replaced by a smooth morphism $f: Z \to X$ of complex manifolds, and $Y \subset X$ is a irreducible closed subvariety with conormal space $C:=T^*_Y(X)$, such that $\omega_f^{-1}(C)$ is irreducible in $Z \times_X f^*(T^*(X))$.
\end{remark}

\begin{example}
Consider the divisor $X_{s_1s_3s_2}^P \subset \Gr(2,4)$.
Consider $\pi: \Fl(4) \to \Gr(2,4)$.
Then $\pi^{-1}X_{s_1s_3s_2}^P = X_{s_1s_2s_3s_2s_1}^B$, and using the equation \eqref{E:Ma312} and Theorem \ref{thm:pbMa} we obtain:
\[
\begin{split}
\cMa(X_{1,2,3,2,1}^B) = & [X_{1, 2, 3, 2, 1}]+3[X_{2, 3, 2, 1}]+3 [X_{1, 2, 3, 1}]+10 [X_{2, 3, 1}] +28 [X_{3, 1}] \\
&+2 [X_{1, 2, 3, 2}] +8 [X_{2, 3, 2}]+4 [X_{1, 2, 3}]+16 [X_{2, 3}]+28 [X_{3}]\\
& +2 [X_{3, 1, 2, 1}] +4 [X_{3, 2, 1}]+8 [X_{1, 2, 1}]+16 [X_{2, 1}]+28 [X_{1}]\\
& +4 [X_{3, 1, 2}]+12 [X_{3, 2}]+12 [X_{1, 2}]+32 [X_{2}]+24 [X_{id}] \/.
\end{split}
\]
(To ease notation, we omitted the $B$ superscript and the $s$'s from the reduced words.)
\end{example}

A calculation involving the Mather class of the pull-back divisor from the Lagrangian Grassmannian $\mathrm{LG}(2,4)$
and its relation to Kazhdan-Lusztig classes,
is given in example \ref{ex:LG24div} below.

\section{Positivity and unimodality of Mather classes}\label{sec:positivity}
\label{s:positivity}
In this section we discuss positivity conjectures for the Euler obstruction and for the Mather class of a Schubert variety,
and we prove them in some cases; we also record a positivity result for Segre-Mather classes.
Finally, we make a unimodality and log concavity conjecture for Mather classes, similar 
to the one for CSM classes \cite{AMSS:specializations}. The proofs are based on 
positivity properties proved in \cite{huh:csm,AMSS:shadows} for CSM classes,
and the results from \cite{boe.fu} and \cite{MR1084458} for the local Euler obstruction. 

\subsection{Positivity conjectures}
In \cite[Rmk. 5.7]{jones:csm}, B. Jones conjectured that all Mather classes for Grassmannians are nonnegative.
Based on substantial experimentation for all cominuscule spaces we make the following conjecture:

\begin{conjecture}[Strong Positivity conjecture of Mather classes]
\label{conj:pos}
Let $X_w^P$ be a Schubert variety in a cominuscule space $G/P$.
Consider the Schubert expansion of the Mather class
\begin{equation}
\label{E:MaSchubExp}
\cMa(X_w^P) = \sum_v a_v [X_v^P] \/.
\end{equation}
Then $a_v > 0$.

More generally, consider the Schubert expansion of the equivariant Mather class 
\[ \cMaT(X_w^P) = \sum_v a_v(t) [X_v^P]_T \/. \]
Then $a_v(t)= a_v(\alpha_1, \ldots , \alpha_r) \in A^*_T(pt)$ is a polynomial in positive simple roots $\alpha_1, \ldots, \alpha_r$ with non-negative coefficients.
\end{conjecture}

We will refer to the situation when $a_v \ge 0$ simply as the `Positivity conjecture', and emphasize `Strong' whenever we can claim it.
A similar positivity conjecture was given in \cite{aluffi.mihalcea:csm,aluffi.mihalcea:eqcsm} for the CSM classes and it was proved in the non-equivariant case in \cite{huh:csm} for Grassmannians and \cite{AMSS:shadows} in general.
Computer evidence suggests a more refined conjecture, for the local Euler obstructions.

\begin{conjecture}
\label{conj:eulerpos}
Let $v,w \in W^P$ such that $v \le w$.
Then the local Euler obstruction $\Eu_{X_w^P}(e_v) \ge 0$.
\end{conjecture}

As we explain below, Boe and Fu \cite{boe.fu} calculated local Euler obstructions for cominuscule spaces
of classical Lie types, and in the process proved Conjecture \ref{conj:eulerpos} for the 
cominuscule spaces of types A, B and D. 

Note that for general spaces the Euler obstruction may be negative.
For instance, if $C$ is a cone over a nonsingular plane curve of degree $d$ with vertex $O$,
then $\Eu_O(C) = 2d-d^2$, cf.~\cite{macpherson:chern}.
Also, the Euler obstruction may be $0$ even for cominuscule spaces - see \cref{ex:EulerLG24,ex:EulerLG36} above.

It is tempting to conjecture the positivity statements above for Schubert varieties in \emph{all} homogeneous spaces $G/P$.
Unfortunately, outside the cominuscule cases, there are very few instances where we have algorithms 
to calculate (non-trivial) Euler obstructions and Mather classes.
By allowing calculations of Mather classes of pull-back Schubert varieties,
\cref{prop:eulerpb,thm:pbMa} provide some evidence on this matter.

\subsection{Kazhdan-Lusztig classes, Mather classes, and positivity} 
Unless otherwise specified, in this section $P$ is an arbitrary parabolic subgroup. 
We refer to \cite[\S 6]{AMSS:shadows} for more details about the material below.
 
Let $IH(X_w^P)$ denote the characteristic cycle of the intersection homology of 
the Schubert variety $X_w^P$.
The characteristic cycle $IH(X_w^P)$ is an effective, conic, Lagrangian cycle in the cotangent bundle 
$T^*(G/P)$.
Its irreducible components are conormal spaces of Schubert varieties;
see e.g. \cite[Thm.~E.3.6]{HTT}.
Therefore there is a expansion 
\begin{equation}
\label{E:mwv}
[IH(X_w^P)]_{T \times \C^*}  = 
\sum_v m_{w,v} [T^*_{X_v^P}(G/P)]_{T \times \C^*} \in A_{\dim G/P}^{T \times \C^*}(T^*(G/P))
\end{equation} 
Define the {\em Kazhdan-Lusztig (KL) class} $KL_w^P \in A_*^T(G/P)$ to be the $\hbar =1$ dehomogenization of 
\[ (-1)^{\ell(w)} (\iota^P)^*[IH(X_w^P)] \in A_0^{T \times \C^*}(G/P) \/.\]
By equation \eqref{E:cmath} this is the same as 
\begin{equation}
\label{E:KL}
KL_w^P =\sum_v (-1)^{\ell(w) - \ell(v) } m_{w,v} \cMaT(X_v^P) \in A_*^{T}(G/P) \/.
\end{equation} 
At the same time, from the proof of the Kazhdan-Lusztig conjectures 
\cite{brylinski.kashiwara:KL,beilinson.bernstein:localisation} and the calculations of $\iota^*$ from 
\cite{AMSS:shadows} (see especially equation (25)
\begin{footnote}{
The results from \cite[\S 6]{AMSS:shadows} are stated for $G/B$,
but everything extends using parabolic versions of the objects considered. 
The formulae for the pull back via the zero section of various characteristic cycles hold for any smooth projective variety; cf.~\cite[Thm.4.3]{AMSS:shadows}.
}\end{footnote}
), it follows that 
\begin{equation}
\label{E:KLP}
KL_w^P = \sum_v P_{w,v}(1) \csmT(X_v^{P,\circ}) \/,
\end{equation} 
where $P_{w,v}$ is the parabolic Kazhdan-Lusztig polynomial;
see e.g. \cite[Prop. 3.4]{deodhar:geometricII}.
(Observe that we could have taken \eqref{E:KLP} to be the definition of the KL class.)

Consider the expansion of the Mather class into (equivariant) CSM classes of Schubert cells:
\begin{equation}
\label{E:cMa}
\cMaT(X_w^P) = \sum_{v}e_{w,v}\csmT(X_v^{P, \circ}) \/.
\end{equation}
Recall that the coefficient $e_{w,v}= \Eu_{X_w^P}(e_v)$
is the local Euler obstruction evaluated at the fixed point $e_v$.
Combining equations \eqref{E:KL}, \eqref{E:KLP} and \eqref{E:cMa},
it follows that the characteristic cycle $IH(X_w^P)$ is irreducible if and only if the local Euler obstruction satisfies
\begin{equation}
\label{E:e=P}
e_{w,v}= P_{w,v}(1)
\end{equation} 
for all $v\in W^P$.
These considerations lead to the following conditional statements.

\begin{proposition}
\label{prop:posconj}
Let $X=G/P$ be a homogeneous space, and let $v,w \in W^P$ such that $v \le w$.
\begin{enumerate}[label=(\alph*)]
\item If \cref{conj:eulerpos} holds for $X$, then the non-equivariant positivity \cref{conj:pos} holds,
i.e.~in \cref{E:MaSchubExp}, the coefficients $a_v \ge 0$. 

\item If $e_{w,v} >0$ for all $v \in W^P$ such that $v \le w$, then the non-equivariant strong positivity \cref{conj:pos} holds for $X$.

\item If the intersection homology characteristic cycle $IH(X_w^P)$ is irreducible, then conjecture \ref{conj:eulerpos} holds.
\end{enumerate}
\end{proposition}
\begin{proof}
Parts (a) and (b) follow from the equation \eqref{E:cMa} (for non-equivariant classes), using that the non-equivariant CSM classes of Schubert cells are nonnegative \cite{huh:csm,AMSS:shadows}, and that the initial term of $\csm(X_v^{P,\circ})$ is $[X_v^P]$.
Part (c) follows from equation \eqref{E:e=P}, using that the Kazhdan-Lusztig polynomials $P_{w,v}$ ($v \le w$) have non-negative integer coefficients, and have constant term equal to $1$.
\end{proof}

The following instances of Conjecture \ref{conj:eulerpos} follow from results of Bressler, Finkelberg and Lunts \cite{MR1084458} in type A, and by Boe and Fu \cite{boe.fu} in classical Lie types.

\begin{theorem}[\cite{MR1084458,boe.fu}]
\label{thm:poseulerobs}
Let $X=G/P$ be a cominuscule space
of classical Lie type A-D except for the Lagrangian Grassmannian $\LG(n,2n)$ for $n \ge 3$.
Then the Euler obstruction $e_{w,v} >0$ if $G$ is of Lie type $A$ or $D$, and $e_{w,v} \ge 0$ in general.
\end{theorem}
\begin{proof}
The strict positivity part follows from \eqref{E:e=P} because 
the Schubert varieties in cominuscule spaces of type A and D have irreducible characteristic cycles.
This is proved by Bressler, Finkelberg and Lunts \cite{MR1084458} in type A,
and by Boe and Fu \cite{boe.fu} in type D (they also reprove the statement for type A).
For the odd quadrics (in type B), Boe and Fu calculated the Euler obstructions explicitly
- see \cite[\S 6.3]{boe.fu}, especially equations (6.3.3) and (6.3.5) - and found them to be non-negative.
Finally, $\mathrm{LG}(2,4)$ is a $3$-dimensional quadric, isomorphic to the type $B_2$ cominuscule space.
This finishes the proof.
\end{proof} 

\begin{corollary}
Let $X=G/P$ be a cominuscule space
of Lie type A -- D, except the Lagrangian Grassmannian $\mathrm{LG}(n,2n)$ for $n \ge 3$.
Let $\pi:G/B \to G/P$ be the natural projection. 

\begin{enumerate}[label=(\alph*)]
\item 
The strong positivity conjecture \ref{conj:pos} holds for all Schubert varieties in $X$ in Lie types A and D, and the weak positivity conjecture holds for the odd quadric in type B.
\item
Let $w \in W^P$.
Then the Mather class $\cMa(\pi^{-1}(X_w^P))\in A_*(G/B)$ has the same (strong/weak) positivity property as $\cMa(X_w^P)$ from part (a).
\end{enumerate}
\end{corollary} 
\begin{proof}
Part (a) follows from \cref{prop:posconj,thm:poseulerobs}.
Part (b) follows because the Euler obstructions for the pull backs $\pi^{-1}(X_w^P)$ coincide with those for $X_w^P$ by \cref{prop:eulerpb}.
This proves (b) and it finishes the proof.
\end{proof}

\begin{remark}
The problem of finding the multiplicities of the characteristic cycle seems to be very difficult. 
Kazhdan and Lusztig \cite{KL:topological} conjectured the irreducibility of characteristic cycles of the IH sheaf in type A.
However, Kashiwara and Tanisaki \cite{kashiwara.tanisaki:characteristic}, then Kashiwara and Saito \cite{kashiwara.saito} found counterexamples for the full flag manifolds of Lie type B and type A respectively.
See also \cite{braden:irred,williamson:reducible} for more about this.
Boe and Fu \cite{boe.fu} also found that the characteristic cycles of the Schubert varieties in cominuscule spaces of Lie types B, C are in general reducible.
\end{remark}

In the next example, we use the methods of this paper to recover an example of Kashiwara and Tanisaki of a reducible IH characteristic cycle. 

\begin{example}\label{ex:LG24div}
Consider the Lagrangian Grassmannian $\mathrm{LG}:=\mathrm{LG}(2,4)$ and the Schubert divisor $X_{1,2}^P \subset \mathrm{LG}(2,4)$
($s_2$ corresponds to the long simple root).
Let $\mathrm{SF}:= \mathrm{SF}(1,2;4)$ be the complete flag manifold of type $C_2$; it parametrizes flags $F_1 \subset F_2 \subset \C^4$ where $F_i$ is isotropic with respect to a symplectic form.
Let $\pi: \mathrm{SF}\to \mathrm{LG}$ be the projection.
This is a $\bP^1$-bundle, and the preimage $\pi^{-1}(X_{1,2}^P)$ is the Schubert divisor indexed by $X_{1,2,1}^B \subset \mathrm{SF}$.
A calculation of the Kazhdan-Lusztig polynomials using e.g. SAGE shows that $P_{121,v}=1$ for any $v \le s_1 s_2 s_1$.
Thus the non-equivariant KL class of $X_{1,2,1}^B$ is:
\[
KL_{1,2,1}^B = \csm(X_{1,2,1}^{B,\circ}) + \csm(X_{1,2}^{B,\circ}) + \csm(X_{2,1}^{B,\circ}) + \csm(X_{2}^{B,\circ}) + \csm(X_{1}^{B,\circ})+ \csm(X_{id}^B) \/.
\]
Using now the calculation for the local Euler obstructions from \cref{ex:EulerLG24} we obtain that
\[
\cMa(X_{1,2}^P) = \csm(X_{1,2}^{P,\circ}) + \csm(X_{2}^{P,\circ}) \/.
\]
Then from \cref{thm:pbMa} and the Verdier-Riemann-Roch Theorem \ref{thm:VRR} 
(or \cref{prop:eulerpb}) it follows that
\[
\begin{split}
\cMa(X_{1,2,1}^B) & = \csm(\pi^{-1}X_{1,2}^{P,\circ}) + \csm(\pi^{-1}X_{2}^{P,\circ}) \\ & = \csm(X_{1,2,1}^{B,\circ}) + \csm(X_{1,2}^{B,\circ}) + \csm(X_{2,1}^{B,\circ}) + \csm(X_{2}^{B,\circ}) \/.
\end{split}
\]
Using that $\cMa(X_1^B) = \csm(X_1^{B,\circ}) + \csm(X_{id}^{B,\circ})$ (as $X_1^B \simeq \bP^1$), we deduce that 
\[ KL_{1,2,1}^B = \cMa(X_{1,2,1}^B) + \cMa(X_{1}^B) \/.\] 

By \cref{thm:segre-mather} and the definition of the KL class, this shows that the IH characteristic cycle $IH(X_{1,2,1}^B) \subset T^*(G/B)$ satisfies
\[ {IH}(X_{1,2,1}^B) = [T^*_{X_{1,2,1}^B}(\mathrm{SF})] + [T^*_{X_{1}^B}(\mathrm{SF})] \/,
\]
in accordance to \cite[p.~194]{kashiwara.tanisaki:characteristic}.
\begin{footnote}
{The example in \cite{kashiwara.tanisaki:characteristic} is in type $B_2$, but the corresponding complete flag variety is isomorphic to the variety $\mathrm{SF}$.}
\end{footnote}
\end{example}
We end by observing that by equation \eqref{E:e=P}, the characteristic cycle $IH(\pi^{-1}(X_w^P))$ must be reducible every time the Euler obstruction $\Eu_{X_w^P}(e_v) \le 0$ for some $v$.
For more such examples, see \cref{ex:EulerLG36}.
It would be interesting to find criteria when this happens and compare with the reducibility criteria from \cite{boe.fu}.

\subsection{Segre-Mather classes are alternating} We also record a positivity result of Segre-Mather classes of Schubert varieties in cominuscule spaces.
This result can be proved in full $G/P$ generality using \cite[Thm.~1.1]{AMSS:ssmpos}.
Here we restrict to the cominuscule case, but we provide a self-contained proof based on \cref{thm:mather}.

Recall the definition of the {\em Segre-Mather} class:
\[
\sma(X_w^P) := \frac{\cMa(X_w^P)}{c(T(G/P))} \quad (w \in W^P) \/.
\]
Our next result shows that these classes are alternating.
\begin{proposition}
Let $G/P$ be a cominuscule space, and let $w \in W^P$.
Consider the expansion of the non-equivariant Segre-Mather class in its homogeneous components:
\[
\sma(X_w^P) = c_0 +c_1+ \ldots \/,
\]
where $c_i \in A_i(G/P)$.
Then $\sma(X_w^P)$ is alternating, i.e. for each component $(-1)^{\ell(w) -i}c_i$ is effective.
\end{proposition}
\begin{proof}
By \cref{thm:mather},
\[
\frac{\cMa(X_w^P)}{c(T(G/P))}= \pi_*(c(\mathcal{U}_w^*)^{-1}\cap [X_w^B]) \/.
\]
But $\mathcal{U}_w^*$ is globally generated, because it is a quotient of the bundle $\pi^* T(G/P)$.
Then the Segre class $c(\mathcal{U}_w^*)^{-1}\cap [X_w^B]$ is alternating.
\end{proof}

\subsection{Unimodality and log concavity of Mather polynomials}
\label{ss:logc}
For $w \in W^P$, consider the Schubert expansion
\[ \cMa(X_w^P) = \sum a_{w,v} [X_v^P] \/.\]
The \emph{Mather polynomial} associated to $w$ is
\[ M_w(x) = \sum a_{w,v} x^{\ell(v)} \/. \]
For instance, the Mather polynomial of the Schubert variety $X_{(4,3,1)} \subset \LG(4,8)$ from \cref{Ex:LG431} is 
\[ M_{(4,3,1)}(x) = x^8+11 x^7+52 x^6+152 x^5+286 x^4+452 x^3+246 x^2+132 x+24 \/. \]
Following \cite{stanley:Log},
we say that a polynomial $a_n x^n + a_{n-1} x^{n-1} + \ldots + a_1 x+ a_0$ is {\em unimodal}
if $a_0 \le a_1 \le \ldots \le a_k \ge a_{k+1} \ge \ldots \ge a_n$ for some index $k$.
It is {\em log concave} if $a_i^2 \ge a_{i-1} a_{i+1}$ for all $i$ (by convention $a_{-i} = a_{n+i} =0$ for all $i\ge 1$).
If one assumes that the polynomial has strictly positive coefficients, then any log concave polynomial is also unimodal. 

Substantial amount of calculations for all Lie types supports the following: 

\begin{conjecture}
\label{conj:logc}
Let $X = G/P$ be a cominuscule space and $w \in W^P$.
\begin{enumerate}[label=(\alph*)]
\item
The Mather polynomial $M_w$ has strictly positive coefficients and it is unimodal.
\item
Assume in addition that $G/P = \Gr(k,n)$.
Then  $M_w$ is log concave.
\end{enumerate}
\end{conjecture}

We checked this conjecture for the Grassmannians $\Gr(k,n)$ where $n \le 8$,
the cominuscule spaces of Lie types $C_n$ and $D_n$ where $n \le 5$,
and all but $5$ Mather classes in the Cayley plane $E_6/P_6$.
The log concavity fails outside type A.
For instance, the Mather polynomial of the $5$ dimensional quadric $\mathrm{OG}(1,7)$ is
\[ x^5+5 x^4+11 x^3+26 x^2+18 x+6 \/.\]
(This Mather class is the same as the total Chern class).
Similarly, the Mather classes of the Lagrangian Grassmannian $\mathrm{LG}(5,10)$ and of the Orthogonal Grassmannian $\mathrm{OG}(4,8)$ are not log concave.

The unimodality and log concavity properties of characteristic classes of singular varieties seem to be new and unexplored phenomena.
For instance, in analogy to the Mather polynomial one may define two flavors of a \emph{CSM polynomial}:
one obtained from the CSM class of a Schubert cell, and the other from the CSM class of a Schubert variety.
This is conjectured to satisfy an analog of \cref{conj:logc};
more details will be discussed in the upcoming note \cite{AMSS:specializations}.
Log concavity has also been conjectured for certain coefficients of motivic Chern classes of Schubert cells \cite[\S 6.2]{FRW:axiomatic}.
It would be interesting to know whether these phenomena fit into the (Hodge-Riemann and Hard Lefschetz) framework from \cite{huh:combinatorial} or \cite{HMMS:log}.

\section{Tables}\label{sec:tables}
In this section, we aggregate our computations of the Chern-Mather classes and local Euler obstructions.
The computations of the Euler obstructions in $\LG(4,8)$ rely on the recurrence relations obtained by Boe and Fu \cite{boe.fu} and have been checked using the results of this note.

\subsection{Schubert Varieties in Cominuscule spaces}
We recall some facts about diagrams indexing the Schubert varieties in cominuscule spaces.
Our main reference is \cite{BCMP:qkchev}.

Let $G/P$ be a cominuscule space 
corresponding to the simple root $\alpha_P$.
Recall from \cref{sec:preliminaries} that the set of positive roots $R^+$ is equipped with a partial 
order $<$. Let \rp be the subset of those roots $\alpha \in R^+$ such that $\alpha \ge \alpha_P$. This    
is a lattice under the partial order $<$.
A {\em lower-order ideal} in \rp\ is a subset satisfying the following condition:
for any pair of elements $i,j\in\rp$ with $i\in I$ and $j\leq i$, we have $j\in I$. 
Following \cite{proctor:bruhat}, the Weyl group elements $w \in W^P$ 
are in bijection with the lower-order ideals in \rp. 

Fix a Hasse diagram for \rp, (cf.~e.g.~\cref{rootPoset}). Then the 
lower-order ideal $I_w$ of $w \in W^P$ gives the {\em diagram of $w$}.
This generalizes the usual Young diagram associated to a Schubert variety 
in the classical Lie types; for an equivalent model using quivers, see \cite{perrin:small-res}.
In particular, the nodes of $I_w$, which are also the boxes of the diagram of $w$, are given by 
positive roots in \rp. These are precisely the positive roots in the inversion set of $w$; thus
the length of $w$ equals the number of boxes in the diagram of $w$. As explained in \cite[\S 3]{BCMP:qkchev}, 
each box in the diagram of $w$ may also be labelled by a simple reflection, and these labels can be used 
to obtain a minimal word for $w$. For instance, the Hasse diagram for the Cayley plane $E_6/P_6$, and the 
associated diagram, are given in the \cref{rootPoset}.
 The shaded parts give the diagram $(5,2,1)$, corresponding to 
the element $w \in W^P$ with a
reduced expression $w=s_5s_4s_2s_1s_3s_4s_5s_6$; see also \cref{ss:examples}.
  
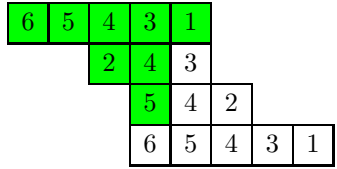
\begin{center}
\begin{tabular}{c}
\ytableausetup{centertableaux,boxsize=normal}
\begin{ytableau}*(green) 6& *(green){5} &  *(green){4} & *(green) {3} & *(green) {1} 
\\ \none &  \none & *(green) {2} & *(green) {4} & 3 \\ \none & \none & \none & *(green) {5} & {4} & 2 
\\ \none & \none & \none & 6 & 5 & 4 & {3} & 1 \end{ytableau} 
\end{tabular}
\end{center}

\subsection{Type A: Grassmannians}
The Schubert subvarieties of $\Gr(k,n)$ are indexed by Young diagrams (or partitions) $\lambda= (\lambda_1 \ge \ldots \ge \lambda_k)$ such that $\lambda_1 \le n-k$ and $\lambda_k \ge 0$.
The Schubert variety $X_\lambda$ has dimension $|\lambda| = \lambda_1 + \ldots + \lambda_k$.

\Cref{TBL:Gr36} lists the Mather classes of Schubert varieties in $\Gr(3,7)$.
The expansion of the Chern-Mather class of a Schubert variety
in terms of ordinary Schubert classes 
is given in the column indexed by the corresponding partition.
Following \cref{sub:stability}, we see that this table also contains the Mather classes of all $\Gr(k,n)$ with $k \le 3$ and $n-k\le 4$.  

\subsection{Type C: Lagrangian Grassmannians}
Let $G/P = \LG(n,2n)$, the variety parametrizing the Lagrangian subspaces of a ${2n}$ dimensional symplectic vector space.
The Schubert subvarieties of $G/P$ are indexed by {\em strict partitions}
$\lambda = (\lambda_1 > \lambda_2 > \ldots > \lambda_k)$, where $\lambda_1 \le n$, $\lambda_k >0$, and $0\le k \le n$.
As before, we have $\dim X_\lambda=|\lambda|$.

\Cref{TBL:cMaLG48} lists the Mather classes of Schubert varieties in $\LG(4,8)$; 
The expansion of the Chern-Mather class of a Schubert variety 
in terms of ordinary Schubert classes 
is given in the column indexed by the corresponding partition. 
\Cref{TBL:eulerLG4} list the Euler obstructions of the Schubert varieties in $\LG(4,8)$.

\subsection{Type \texorpdfstring{$E_6$}{E6}: the Cayley plane}
\Cref{TBL:cMaE6} lists the Mather classes of some Schubert subvarieties of the cominuscule space $E_6/P_6$; 
The expansion of the Chern-Mather class of a Schubert variety 
in terms of ordinary Schubert classes 
is given in the column indexed by the corresponding partition.

\subsection{Stability}
\label{sub:stability}
Our `homological' indexing conventions for Schubert varieties gives a stability property for the Mather classes in the ordinary 
Grassmannians. and for the maximal isotropic Grassmannians in type C and D. We explain this for the 
ordinary Grassmannian, following \cite[\S 2.1]{aluffi.mihalcea:csm}.

Fix a partition $\lambda = (\lambda_1, \ldots , \lambda_k)$ included in the $k \times (n-k)$ rectangle. 
Fix the standard flag $F_\bullet: F_1 \subset \ldots \subset F_n = \C^n$ where 
$F_i = \langle e_1, \ldots , e_i \rangle$. The Schubert 
variety $X_\lambda \subset \Gr(k,n)$ is defined by 
\[ X_\lambda = \{ V : \dim V \cap F_{\lambda_{k+1- i} +i} \ge i \} \/. \] 
If the $k \times (n-k)$ diagram is included in the $k' \times (n'-k')$ diagram then one can define an embedding
$i: \Gr(k,n) \hookrightarrow \Gr(k',n')$ by $i(V) = \langle e_1 , \ldots , e_{k'-k} \rangle \oplus \widetilde{V}$, where 
$\widetilde{V}$ is obtained from $V$ by shifting the indices of basis elements according to the rule 
$e_j \mapsto e_{j+k'-k}$. With this definition, 
$i(X_\lambda) = X_\lambda$, and it follows that 
\begin{equation*}
i_* \cMa(X_\lambda) = \cMa(X_\lambda) \in A_*(\Gr(k',n')) \/.
\end{equation*}
For instance, the Schubert variety $X_\Box \subset \Gr(1,3)$ is $\left\{\langle ae_1 + be_2 \rangle \subset \C^3\mid[a:b] \in \bP^1\right\}$. 
Under the inclusion $\Gr(1,3) \hookrightarrow \Gr(3,7)$ the image of $X_\Box$ is the Schubert variety parametrizing the dimension $3$ subspaces $\langle e_1, e_2, ae_3+be_4 \rangle \in \Gr(3,7)$.  


One may define similar embeddings, $i: \LG(n,2n) \hookrightarrow \LG(n',2n')$, and 
$i: \OG(n,2n) \hookrightarrow \OG(n',2n')$ with $n\leq n'$.
In all these cases, we have have $i(X_\lambda)=X_\lambda$, (cf.~\cite[\S6.2,\S7.2]{smt},
and hence $i_* \cMa(X_\lambda) = \cMa(X_\lambda)$.
We leave it to the reader to check the 
details in other types.

\bibliography{conormal}
\bibliographystyle{halpha}

\newpage
\begin{table}
\centering
\captionsetup[subfloat]{labelformat=empty}

\subfloat[][]{
\noindent\makebox[\textwidth]{
\begin{tabular}{|c|c|c|c|c|c|c|c|c|c|c|c|c|c|c|c|c|c|}
\hline
     & () & 1 & 2 & 21 & 3 & 4  & 31 & 41 & 32 & 42 & 321 & 43  & 421 & 431 & 432 & 4321 \\ \hline
()   & 1  & 2 & 2 & 4  & 4 & 4  & 8  & 8  & 4  & 8  & 8   & 12  & 16  & 24  & 8   & 16   \\
1    & 0  & 1 & 4 & 8  & 9 & 16 & 20 & 34 & 18 & 40 & 36  & 64  & 80  & 132 & 64  & 128  \\
2    & 0  & 0 & 1 & 3  & 5 & 14 & 14 & 37 & 23 & 64 & 46  & 114 & 128 & 246 & 172 & 344  \\
21   & 0  & 0 & 0 & 1  & 0 & 0  & 5  & 14 & 18 & 58 & 37  & 114 & 120 & 269 & 268 & 536  \\
3    & 0  & 0 & 0 & 0  & 1 & 6  & 3  & 17 & 7  & 36 & 15  & 80  & 76  & 183 & 176 & 352  \\
4    & 0  & 0 & 0 & 0  & 0 & 1  & 0  & 3  & 0  & 7  & 0   & 19  & 15  & 45  & 52  & 105  \\
31   & 0  & 0 & 0 & 0  & 0 & 0  & 1  & 6  & 6  & 34 & 15  & 90  & 82  & 241 & 336 & 674  \\
41   & 0  & 0 & 0 & 0  & 0 & 0  & 0  & 1  & 0  & 6  & 0   & 21  & 15  & 60  & 102 & 210  \\
32   & 0  & 0 & 0 & 0  & 0 & 0  & 0  & 0  & 1  & 6  & 4   & 25  & 23  & 92  & 190 & 386  \\
42   & 0  & 0 & 0 & 0  & 0 & 0  & 0  & 0  & 0  & 1  & 0   & 7   & 4   & 27  & 68  & 147  \\
321  & 0  & 0 & 0 & 0  & 0 & 0  & 0  & 0  & 0  & 0  & 1   & 0   & 6   & 25  & 88  & 184  \\
43   & 0  & 0 & 0 & 0  & 0 & 0  & 0  & 0  & 0  & 0  & 0   & 1   & 0   & 4   & 14  & 34   \\
421  & 0  & 0 & 0 & 0  & 0 & 0  & 0  & 0  & 0  & 0  & 0   & 0   & 1   & 7   & 32  & 76   \\
431  & 0  & 0 & 0 & 0  & 0 & 0  & 0  & 0  & 0  & 0  & 0   & 0   & 0   & 1   & 8   & 24   \\
432  & 0  & 0 & 0 & 0  & 0 & 0  & 0  & 0  & 0  & 0  & 0   & 0   & 0   & 0   & 1   & 5    \\
4321 & 0  & 0 & 0 & 0  & 0 & 0  & 0  & 0  & 0  & 0  & 0   & 0   & 0   & 0   & 0   & 1    \\
\hline
\end{tabular}
}
}
\caption{ Chern-Mather classes for $\LG(4,8)$.  }
\label{TBL:cMaLG48}
\end{table}

\begin{table}
\centering
\captionsetup[subfloat]{labelformat=empty}

\subfloat[][]{
\noindent\makebox[\textwidth]{
\begin{tabular}{|c|c|c|c|c|c|c|c|c|c|c|c|c|c|c|c|c|}
\hline
     & () & 1 & 2 & 21 & 3 & 31 & 32 & 321 & 4 & 41 & 42 & 421 & 43 & 431 & 432 & 4321\\ \hline
()   & 1  & 1 & 0 & 1  & 1 & 2  & 1  & 1   & 0 & 1  & 2  & 2   & 2  & 3   & 0   & 1\\
1    & 0  & 1 & 1 & 1  & 1 & 2  & 0  & 1   & 1 & 1  & 0  & 2   & 1  & 3   & 1   & 1\\
2    & 0  & 0 & 1 & 1  & 1 & 1  & 0  & 1   & 1 & 1  & 0  & 2   & 1  & 2   & 1   & 1\\
21   & 0  & 0 & 0 & 1  & 0 & 1  & 1  & 1   & 0 & 1  & 2  & 2   & 1  & 2   & 0   & 1\\
3    & 0  & 0 & 0 & 0  & 1 & 1  & 0  & 1   & 1 & 1  & 0  & 1   & 1  & 2   & 1   & 1\\
31   & 0  & 0 & 0 & 0  & 0 & 1  & 1  & 1   & 0 & 1  & 1  & 1   & 1  & 2   & 0   & 1\\
32   & 0  & 0 & 0 & 0  & 0 & 0  & 1  & 1   & 0 & 0  & 1  & 1   & 1  & 1   & 0   & 1\\
321  & 0  & 0 & 0 & 0  & 0 & 0  & 0  & 1   & 0 & 0  & 0  & 1   & 0  & 1   & 1   & 1\\
4    & 0  & 0 & 0 & 0  & 0 & 0  & 0  & 0   & 1 & 1  & 0  & 1   & 1  & 2   & 1   & 1\\
41   & 0  & 0 & 0 & 0  & 0 & 0  & 0  & 0   & 0 & 1  & 1  & 1   & 1  & 2   & 0   & 1\\
42   & 0  & 0 & 0 & 0  & 0 & 0  & 0  & 0   & 0 & 0  & 1  & 1   & 1  & 1   & 0   & 1\\
421  & 0  & 0 & 0 & 0  & 0 & 0  & 0  & 0   & 0 & 0  & 0  & 1   & 0  & 1   & 1   & 1\\
43   & 0  & 0 & 0 & 0  & 0 & 0  & 0  & 0   & 0 & 0  & 0  & 0   & 1  & 1   & 0   & 1\\
431  & 0  & 0 & 0 & 0  & 0 & 0  & 0  & 0   & 0 & 0  & 0  & 0   & 0  & 1   & 1   & 1\\
432  & 0  & 0 & 0 & 0  & 0 & 0  & 0  & 0   & 0 & 0  & 0  & 0   & 0  & 0   & 1   & 1\\
4321 & 0  & 0 & 0 & 0  & 0 & 0  & 0  & 0   & 0 & 0  & 0  & 0   & 0  & 0   & 0   & 1\\
\hline
\end{tabular}
}
}
\caption{
Local Euler obstructions for $\LG(4,8)$.
The Euler obstruction of the Schubert variety $X_v$ at the point $u$ is given in row $u$ and column $v$.
}
\label{TBL:eulerLG4}
\end{table}

\begin{table}
\centering
\captionsetup[subfloat]{labelformat=empty}
\subfloat[][]{
\noindent\makebox[\textwidth]{
\begin{tabular}{|c|c|c|c|c|c|c|c|c|c|c|c|c|c|c|c|c|c|c|c|c|}
\hline
&()&1&2&11&3&21&111&31&22&211&32&311&221&33&321&222&331&322&332&333\\\hline
()&1&2&3&3&4&6&4&8&6&8&12&12&12&10&24&10&20&20&30&20\\
1&0&1&3&3&6&8&6&15&12&15&27&27&27&30&60&30&66&66&108&90\\
2&0&0&1&0&4&3&0&11&7&6&23&21&17&35&54&25&82&74&144&150\\
11&0&0&0&1&0&3&4&6&7&11&17&21&23&25&54&35&74&82&144&150\\
3&0&0&0&0&1&0&0&3&0&0&7&6&0&15&17&0&37&25&69&90\\
21&0&0&0&0&0&1&0&4&4&4&15&15&15&30&52&30&98&98&210&270\\
111&0&0&0&0&0&0&1&0&0&3&0&6&7&0&17&15&25&37&69&90\\
31&0&0&0&0&0&0&0&1&0&0&4&4&0&11&15&0&40&30&93&146\\
22&0&0&0&0&0&0&0&0&1&0&4&0&4&12&15&12&42&42&108&174\\
211&0&0&0&0&0&0&0&0&0&1&0&4&4&0&15&11&30&40&93&146\\
32&0&0&0&0&0&0&0&0&0&0&1&0&0&5&4&0&19&12&54&108\\
311&0&0&0&0&0&0&0&0&0&0&0&1&0&0&4&0&11&11&36&66\\
221&0&0&0&0&0&0&0&0&0&0&0&0&1&0&4&5&12&19&54&108\\
33&0&0&0&0&0&0&0&0&0&0&0&0&0&1&0&0&4&0&12&32\\
321&0&0&0&0&0&0&0&0&0&0&0&0&0&0&1&0&5&5&24&58\\
222&0&0&0&0&0&0&0&0&0&0&0&0&0&0&0&1&0&4&12&32\\
331&0&0&0&0&0&0&0&0&0&0&0&0&0&0&0&0&1&0&5&17\\
322&0&0&0&0&0&0&0&0&0&0&0&0&0&0&0&0&0&1&5&17\\
332&0&0&0&0&0&0&0&0&0&0&0&0&0&0&0&0&0&0&1&6\\
333&0&0&0&0&0&0&0&0&0&0&0&0&0&0&0&0&0&0&0&1\\
\hline
\end{tabular}
}
}
\caption{ Chern-Mather classes for $\Gr(3,6)$.  }
\label{TBL:Gr36}
\end{table}

\begin{table}
\centering
\captionsetup[subfloat]{labelformat=empty}
\subfloat[][]{
\noindent\makebox[\textwidth]{
\begin{tabular}{|c|c|c|c|c|c|c|c|c|c|c|c|c|c|c|c|}
\hline
&1111&2111&3111&2211&3211&2221&2222&3311&3221&3222&3321&3322&3331&3332&3333\\\hline
()&5&10&15&15&30&20&15&30&40&30&60&45&40&60&35\\
1&10&24&42&42&92&64&60&108&140&130&228&210&190&300&210\\
2&0&10&34&29&89&56&65&141&163&180&315&341&330&544&455\\
11&10&26&48&51&117&88&105&153&203&240&351&411&360&624&525\\
3&0&0&10&0&29&0&0&66&56&65&155&185&205&351&350\\
21&0&10&36&36&119&82&120&216&258&360&541&738&690&1266&1260\\
111&5&14&27&30&71&58&90&99&139&215&251&387&305&621&630\\
31&0&0&10&0&36&0&0&94&82&120&252&375&397&768&896\\
22&0&0&0&10&36&36&67&97&121&214&300&502&483&968&1141\\
211&0&5&19&19&67&49&91&128&166&297&369&652&547&1193&1407\\
32&0&0&0&0&10&0&0&46&36&67&157&281&318&678&938\\
311&0&0&5&0&19&0&0&51&49&91&157&298&278&645&868\\
221&0&0&0&5&19&24&58&54&86&199&229&505&434&1056&1470\\
33&0&0&0&0&0&0&0&10&0&0&36&67&100&228&376\\
321&0&0&0&0&5&0&0&24&24&58&110&257&257&672&1076\\
222&0&0&0&0&0&5&18&0&19&65&54&174&136&408&680\\
331&0&0&0&0&0&0&0&5&0&0&24&58&80&224&427\\
322&0&0&0&0&0&0&0&0&5&18&24&83&78&257&497\\
332&0&0&0&0&0&0&0&0&0&0&5&18&29&101&238\\
333&0&0&0&0&0&0&0&0&0&0&0&0&5&18&56\\
1111&1&3&6&7&17&15&31&25&37&77&69&145&90&245&301\\
2111&0&1&4&4&15&11&26&30&40&93&93&218&146&423&588\\
3111&0&0&1&0&4&0&0&11&11&26&36&88&66&198&302\\
2211&0&0&0&1&4&5&16&12&19&59&54&163&108&368&604\\
3211&0&0&0&0&1&0&0&5&5&16&24&75&58&207&378\\
2221&0&0&0&0&0&1&6&0&4&23&12&66&32&168&336\\
2222&0&0&0&0&0&0&1&0&0&4&0&12&0&32&80\\
3311&0&0&0&0&0&0&0&1&0&0&5&16&17&65&141\\
3221&0&0&0&0&0&0&0&0&1&6&5&29&17&95&215\\
3222&0&0&0&0&0&0&0&0&0&1&0&5&0&17&49\\
3321&0&0&0&0&0&0&0&0&0&0&1&6&6&35&98\\
3322&0&0&0&0&0&0&0&0&0&0&0&1&0&6&23\\
3331&0&0&0&0&0&0&0&0&0&0&0&0&1&6&24\\
3332&0&0&0&0&0&0&0&0&0&0&0&0&0&1&7\\
3333&0&0&0&0&0&0&0&0&0&0&0&0&0&0&1\\
\hline
\end{tabular}
}
}
\caption{ Chern-Mather classes for $\Gr(3,7)$.  }
\label{TBL:Gr37}
\end{table}

\begin{table}
\centering
\captionsetup[subfloat]{labelformat=empty}
\subfloat[][]{
\noindent\makebox[\textwidth]{%
\begin{tabular}{|c|c|c|c|c|c|c|c|c|c|c|c|c|c|c|c|c|c|c|c|c|}
\hline
     & 4111 & 4211 & 4311 & 4221 & 4222 & 4411 & 4321 & 4322 \\ \hline
()   & 20   & 40   & 60   & 60   & 45   & 45   & 120  & 90   \\
1    & 64   & 140  & 228  & 228  & 210  & 210  & 480  & 440  \\
2    & 76   & 188  & 345  & 333  & 355  & 405  & 762  & 808  \\
11   & 76   & 188  & 333  & 345  & 405  & 355  & 762  & 888  \\
3    & 44   & 118  & 243  & 219  & 245  & 375  & 556  & 636  \\
21   & 84   & 264  & 555  & 555  & 750  & 750  & 1370 & 1830 \\
111  & 44   & 118  & 219  & 243  & 375  & 245  & 556  & 856  \\
31   & 46   & 155  & 374  & 340  & 480  & 653  & 969  & 1383 \\
22   & 0    & 84   & 272  & 272  & 466  & 466  & 814  & 1324 \\
211  & 46   & 155  & 340  & 374  & 653  & 480  & 969  & 1685 \\
32   & 0    & 46   & 201  & 157  & 281  & 453  & 644  & 1102 \\
311  & 24   & 86   & 215  & 215  & 388  & 388  & 644  & 1182 \\
221  & 0    & 46   & 157  & 201  & 453  & 281  & 644  & 1386 \\
33   & 0    & 0    & 46   & 0    & 0    & 153  & 157  & 281  \\
321  & 0    & 24   & 110  & 110  & 257  & 257  & 478  & 1074 \\
222  & 0    & 0    & 0    & 46   & 153  & 0    & 157  & 493  \\
331  & 0    & 0    & 24   & 0    & 0    & 83   & 110  & 257  \\
322  & 0    & 0    & 0    & 24   & 83   & 0    & 110  & 363  \\
332  & 0    & 0    & 0    & 0    & 0    & 0    & 24   & 83   \\
4    & 10   & 29   & 66   & 56   & 65   & 138  & 155  & 185  \\
41   & 10   & 36   & 94   & 82   & 120  & 216  & 252  & 375  \\
42   & 0    & 10   & 46   & 36   & 67   & 140  & 157  & 281  \\
411  & 5    & 19   & 51   & 49   & 91   & 119  & 157  & 298  \\
43   & 0    & 0    & 10   & 0    & 0    & 56   & 36   & 67   \\
421  & 0    & 5    & 24   & 24   & 58   & 75   & 110  & 257  \\
44   & 0    & 0    & 0    & 0    & 0    & 10   & 0    & 0    \\
431  & 0    & 0    & 5    & 0    & 0    & 29   & 24   & 58   \\
422  & 0    & 0    & 0    & 5    & 18   & 0    & 24   & 83   \\
441  & 0    & 0    & 0    & 0    & 0    & 5    & 0    & 0    \\
432  & 0    & 0    & 0    & 0    & 0    & 0    & 5    & 18   \\
1111 & 10   & 29   & 56   & 66   & 138  & 65   & 155  & 327  \\
2111 & 10   & 36   & 82   & 94   & 216  & 120  & 252  & 587  \\
3111 & 5    & 19   & 49   & 51   & 119  & 91   & 157  & 378  \\
2211 & 0    & 10   & 36   & 46   & 140  & 67   & 157  & 465  \\
3211 & 0    & 5    & 24   & 24   & 75   & 58   & 110  & 335  \\
2221 & 0    & 0    & 0    & 10   & 56   & 0    & 36   & 193  \\
2222 & 0    & 0    & 0    & 0    & 10   & 0    & 0    & 36   \\
3311 & 0    & 0    & 5    & 0    & 0    & 18   & 24   & 75   \\
3221 & 0    & 0    & 0    & 5    & 29   & 0    & 24   & 134  \\
3222 & 0    & 0    & 0    & 0    & 5    & 0    & 0    & 24   \\
3321 & 0    & 0    & 0    & 0    & 0    & 0    & 5    & 29   \\
3322 & 0    & 0    & 0    & 0    & 0    & 0    & 0    & 5    \\
4111 & 1    & 4    & 11   & 11   & 26   & 26   & 36   & 88   \\
4211 & 0    & 1    & 5    & 5    & 16   & 16   & 24   & 75   \\
4311 & 0    & 0    & 1    & 0    & 0    & 6    & 5    & 16   \\
4221 & 0    & 0    & 0    & 1    & 6    & 0    & 5    & 29   \\
4222 & 0    & 0    & 0    & 0    & 1    & 0    & 0    & 5    \\
4411 & 0    & 0    & 0    & 0    & 0    & 1    & 0    & 0    \\
4321 & 0    & 0    & 0    & 0    & 0    & 0    & 1    & 6    \\
4322 & 0    & 0    & 0    & 0    & 0    & 0    & 0    & 1    \\
\hline
\end{tabular}
}
}
\caption{ Chern-Mather classes for $\Gr(4,8)$.  }
\end{table}

\begin{table}
\centering
\captionsetup[subfloat]{labelformat=empty}
\subfloat[][]{
\noindent\makebox[\textwidth]{%
\begin{tabular}{|c|c|c|c|c|c|c|c|c|c|c|c|c|c|c|c|c|c|c|c|c|}
\hline
     & 4421 & 4331 & 4422 & 4332 & 4431 & 4333 & 4432 & 4441 & 4433  & 4442  & 4443  & 4444  \\ \hline
()   & 90   & 80   & 90   & 120  & 120  & 70   & 180  & 70   & 105   & 105   & 140   & 70    \\
1    & 440  & 400  & 480  & 630  & 630  & 440  & 990  & 440  & 690   & 690   & 960   & 560   \\
2    & 888  & 786  & 1032 & 1286 & 1356 & 1060 & 2208 & 1140 & 1805  & 1855  & 2680  & 1820  \\
11   & 808  & 786  & 1032 & 1356 & 1286 & 1140 & 2208 & 1060 & 1855  & 1805  & 2680  & 1820  \\
3    & 856  & 684  & 1044 & 1152 & 1386 & 1090 & 2322 & 1410 & 2170  & 2370  & 3544  & 2800  \\
21   & 1830 & 1734 & 2520 & 3144 & 3144 & 3090 & 5652 & 3090 & 5520  & 5520  & 8568  & 6720  \\
111  & 636  & 684  & 1044 & 1386 & 1152 & 1410 & 2322 & 1090 & 2370  & 2170  & 3544  & 2800  \\
31   & 1685 & 1439 & 2457 & 2715 & 3097 & 3022 & 5787 & 3624 & 6383  & 6766  & 10942 & 9863  \\
22   & 1324 & 1310 & 2066 & 2578 & 2578 & 3002 & 4996 & 3002 & 5763  & 5763  & 9532  & 8582  \\
211  & 1383 & 1439 & 2457 & 3097 & 2715 & 3624 & 5787 & 3022 & 6766  & 6383  & 10942 & 9863  \\
32   & 1386 & 1234 & 2292 & 2547 & 2943 & 3366 & 5964 & 4070 & 7744  & 8220  & 14264 & 14672 \\
311  & 1182 & 1094 & 2208 & 2466 & 2466 & 3198 & 5502 & 3198 & 7130  & 7130  & 12760 & 12992 \\
221  & 1102 & 1234 & 2292 & 2943 & 2547 & 4070 & 5964 & 3366 & 8220  & 7744  & 14264 & 14672 \\
33   & 493  & 410  & 850  & 903  & 1194 & 1422 & 2551 & 1990 & 3825  & 4263  & 7868  & 9246  \\
321  & 1074 & 1074 & 2344 & 2706 & 2706 & 4176 & 6672 & 4176 & 10200 & 10200 & 19768 & 22900 \\
222  & 281  & 410  & 850  & 1194 & 903  & 1990 & 2551 & 1422 & 4263  & 3825  & 7868  & 9246  \\
331  & 363  & 349  & 820  & 943  & 1090 & 1730 & 2841 & 2038 & 5003  & 5281  & 10942 & 14408 \\
322  & 257  & 349  & 820  & 1090 & 943  & 2038 & 2841 & 1730 & 5281  & 5003  & 10942 & 14408 \\
332  & 83   & 134  & 275  & 446  & 446  & 1014 & 1412 & 1014 & 3084  & 3084  & 7304  & 10946 \\
333  & 0    & 24   & 0    & 83   & 83   & 248  & 275  & 248  & 784   & 784   & 2112  & 3656  \\
4    & 327  & 205  & 415  & 351  & 559  & 350  & 959  & 692  & 985   & 1198  & 1850  & 1701  \\
41   & 587  & 397  & 900  & 768  & 1146 & 896  & 2217 & 1592 & 2658  & 3102  & 5216  & 5376  \\
42   & 465  & 318  & 814  & 678  & 1069 & 938  & 2257 & 1754 & 3163  & 3721  & 6764  & 7926  \\
411  & 378  & 278  & 737  & 645  & 825  & 868  & 1918 & 1232 & 2659  & 2895  & 5388  & 6126  \\
43   & 193  & 100  & 348  & 228  & 510  & 376  & 1131 & 1028 & 1798  & 2310  & 4482  & 6016  \\
421  & 335  & 257  & 768  & 672  & 897  & 1076 & 2322 & 1600 & 3782  & 4152  & 8442  & 10928 \\
44   & 36   & 0    & 67   & 0    & 100  & 0    & 228  & 256  & 376   & 604   & 1236  & 1909  \\
431  & 134  & 80   & 315  & 224  & 429  & 427  & 1167 & 951  & 2157  & 2608  & 5690  & 8421  \\
422  & 75   & 78   & 254  & 257  & 291  & 497  & 939  & 613  & 1855  & 1924  & 4466  & 6551  \\
441  & 24   & 0    & 58   & 0    & 80   & 0    & 224  & 228  & 427   & 656   & 1508  & 2552  \\
432  & 29   & 29   & 101  & 101  & 163  & 238  & 547  & 444  & 1252  & 1450  & 3656  & 6160  \\
442  & 5    & 0    & 18   & 0    & 29   & 0    & 101  & 109  & 238   & 373   & 993   & 1924  \\
433  & 0    & 5    & 0    & 18   & 29   & 56   & 101  & 112  & 304   & 376   & 1088  & 2144  \\
443  & 0    & 0    & 0    & 0    & 5    & 0    & 18   & 34   & 56    & 119   & 360   & 832   \\
444  & 0    & 0    & 0    & 0    & 0    & 0    & 0    & 5    & 0     & 18    & 56    & 160   \\
\hline
\end{tabular}
}
}
\caption{(cont.) Chern-Mather classes for $\Gr(4,8)$.  }
\end{table}

\begin{table}
\centering
\captionsetup[subfloat]{labelformat=empty}
\subfloat[][]{
\noindent\makebox[\textwidth]{%
\begin{tabular}{|c|c|c|c|c|c|c|c|c|c|c|c|c|c|c|c|c|c|c|c|c|}
\hline
     & 4421 & 4331 & 4422 & 4332 & 4431 & 4333 & 4432 & 4441 & 4433  & 4442  & 4443  & 4444  \\ \hline
1111 & 185  & 205  & 415  & 559  & 351  & 692  & 959  & 350  & 1198  & 985   & 1850  & 1701  \\
2111 & 375  & 397  & 900  & 1146 & 768  & 1592 & 2217 & 896  & 3102  & 2658  & 5216  & 5376  \\
3111 & 298  & 278  & 737  & 825  & 645  & 1232 & 1918 & 868  & 2895  & 2659  & 5388  & 6126  \\
2211 & 281  & 318  & 814  & 1069 & 678  & 1754 & 2257 & 938  & 3721  & 3163  & 6764  & 7926  \\
3211 & 257  & 257  & 768  & 897  & 672  & 1600 & 2322 & 1076 & 4152  & 3782  & 8442  & 10928 \\
2221 & 67   & 100  & 348  & 510  & 228  & 1028 & 1131 & 376  & 2310  & 1798  & 4482  & 6016  \\
2222 & 0    & 0    & 67   & 100  & 0    & 256  & 228  & 0    & 604   & 376   & 1236  & 1909  \\
3311 & 83   & 78   & 254  & 291  & 257  & 613  & 939  & 497  & 1924  & 1855  & 4466  & 6551  \\
3221 & 58   & 80   & 315  & 429  & 224  & 951  & 1167 & 427  & 2608  & 2157  & 5690  & 8421  \\
3222 & 0    & 0    & 58   & 80   & 0    & 228  & 224  & 0    & 656   & 427   & 1508  & 2552  \\
3321 & 18   & 29   & 101  & 163  & 101  & 444  & 547  & 238  & 1450  & 1252  & 3656  & 6160  \\
3322 & 0    & 0    & 18   & 29   & 0    & 109  & 101  & 0    & 373   & 238   & 993   & 1924  \\
3331 & 0    & 5    & 0    & 29   & 18   & 112  & 101  & 56   & 376   & 304   & 1088  & 2144  \\
3332 & 0    & 0    & 0    & 5    & 0    & 34   & 18   & 0    & 119   & 56    & 360   & 832   \\
3333 & 0    & 0    & 0    & 0    & 0    & 5    & 0    & 0    & 18    & 0     & 56    & 160   \\
4111 & 88   & 66   & 222  & 198  & 198  & 302  & 600  & 302  & 946   & 946   & 1962  & 2416  \\
4211 & 75   & 58   & 231  & 207  & 207  & 378  & 736  & 378  & 1374  & 1374  & 3168  & 4458  \\
4311 & 29   & 17   & 91   & 65   & 95   & 141  & 356  & 215  & 754   & 830   & 2074  & 3336  \\
4221 & 16   & 17   & 91   & 95   & 65   & 215  & 356  & 141  & 830   & 754   & 2074  & 3336  \\
4222 & 0    & 0    & 16   & 17   & 0    & 49   & 65   & 0    & 197   & 141   & 517   & 944   \\
4411 & 5    & 0    & 16   & 0    & 17   & 0    & 65   & 49   & 141   & 197   & 517   & 944   \\
4321 & 6    & 6    & 35   & 35   & 35   & 98   & 198  & 98   & 542   & 542   & 1656  & 3070  \\
4322 & 0    & 0    & 6    & 6    & 0    & 23   & 35   & 0    & 132   & 98    & 425   & 904   \\
4421 & 1    & 0    & 6    & 0    & 6    & 0    & 35   & 23   & 98    & 132   & 425   & 904   \\
4331 & 0    & 1    & 0    & 6    & 6    & 24   & 35   & 24   & 136   & 136   & 512   & 1128  \\
4422 & 0    & 0    & 1    & 0    & 0    & 0    & 6    & 0    & 23    & 23    & 106   & 262   \\
4332 & 0    & 0    & 0    & 1    & 0    & 7    & 6    & 0    & 41    & 24    & 160   & 416   \\
4431 & 0    & 0    & 0    & 0    & 1    & 0    & 6    & 7    & 24    & 41    & 160   & 416   \\
4333 & 0    & 0    & 0    & 0    & 0    & 1    & 0    & 0    & 6     & 0     & 24    & 80    \\
4432 & 0    & 0    & 0    & 0    & 0    & 0    & 1    & 0    & 7     & 7     & 48    & 152   \\
4441 & 0    & 0    & 0    & 0    & 0    & 0    & 0    & 1    & 0     & 6     & 24    & 80    \\
4433 & 0    & 0    & 0    & 0    & 0    & 0    & 0    & 0    & 1     & 0     & 7     & 31    \\
4442 & 0    & 0    & 0    & 0    & 0    & 0    & 0    & 0    & 0     & 1     & 7     & 31    \\
4443 & 0    & 0    & 0    & 0    & 0    & 0    & 0    & 0    & 0     & 0     & 1     & 8     \\
4444 & 0    & 0    & 0    & 0    & 0    & 0    & 0    & 0    & 0     & 0     & 0     & 1     \\
\hline
\end{tabular}
}
}
\caption{(cont.) Chern-Mather classes for $\Gr(4,8)$.  }
\end{table}

%
\begin{table}
\centering
\captionsetup[subfloat]{labelformat=empty}
\subfloat[][]{
\begin{tabular}{|c|c|c|c|c|c|c|c|c|c|c|c|c|c|}
\hline
     & () & 1 & 2 & 3 & 31 & 4  & 5  & 41 & 51 & 42 & 52  & 421 & 53 \\ \hline
()   & 1  & 2 & 3 & 4 & 5  & 5  & 6  & 10 & 12 & 10 & 18  & 10  & 14\\
1    & 0  & 1 & 3 & 6 & 10 & 10 & 15 & 22 & 33 & 26 & 54  & 32  & 52\\
2    & 0  & 0 & 1 & 4 & 10 & 10 & 20 & 28 & 55 & 44 & 106 & 68  & 130\\
3    & 0  & 0 & 0 & 1 & 5  & 5  & 15 & 22 & 60 & 48 & 144 & 92  & 225\\
31   & 0  & 0 & 0 & 0 & 1  & 0  & 0  & 5  & 15 & 16 & 55  & 40  & 108\\
4    & 0  & 0 & 0 & 0 & 0  & 1  & 6  & 5  & 27 & 16 & 79  & 40  & 159\\
5    & 0  & 0 & 0 & 0 & 0  & 0  & 1  & 0  & 5  & 0  & 16  & 0   & 42\\
41   & 0  & 0 & 0 & 0 & 0  & 0  & 0  & 1  & 6  & 6  & 33  & 22  & 86\\
51   & 0  & 0 & 0 & 0 & 0  & 0  & 0  & 0  & 1  & 0  & 6   & 0   & 22\\
42   & 0  & 0 & 0 & 0 & 0  & 0  & 0  & 0  & 0  & 1  & 6   & 7   & 23\\
52   & 0  & 0 & 0 & 0 & 0  & 0  & 0  & 0  & 0  & 0  & 1   & 0   & 7\\
421  & 0  & 0 & 0 & 0 & 0  & 0  & 0  & 0  & 0  & 0  & 0   & 1   & 0\\
53   & 0  & 0 & 0 & 0 & 0  & 0  & 0  & 0  & 0  & 0  & 0   & 0   & 1\\
\hline
\end{tabular}
}

\subfloat[][]{
\begin{tabular}{|c|c|c|c|c|c|c|c|c|c|}
\hline
      & 521 & 4211 & 531 & 5211 & 532 & 5311 & 533  & 5321 & 5322\\ \hline
()    & 19  & 10   & 28  & 20   & 23  & 30   & 18   & 40   & 27\\
1     & 68  & 40   & 108 & 85   & 100 & 135  & 96   & 190  & 150\\
2     & 164 & 100  & 288 & 240  & 310 & 421  & 360  & 644  & 600\\
3     & 274 & 160  & 552 & 470  & 705 & 948  & 960  & 1611 & 1770\\
31    & 138 & 86   & 318 & 292  & 488 & 672  & 768  & 1276 & 1659\\
4     & 184 & 86   & 450 & 372  & 691 & 912  & 1080 & 1760 & 2271\\
5     & 40  & 0    & 132 & 86   & 242 & 298  & 432  & 680  & 1020\\
41    & 112 & 62   & 322 & 296  & 604 & 840  & 1092 & 1840 & 2814\\
51    & 22  & 0    & 94  & 62   & 216 & 272  & 456  & 734  & 1308\\
42    & 39  & 29   & 139 & 151  & 326 & 507  & 690  & 1285 & 2334\\
52    & 7   & 0    & 46  & 29   & 138 & 180  & 354  & 601  & 1284\\
421   & 6   & 8    & 23  & 45   & 70  & 162  & 180  & 465  & 1014\\
53    & 0   & 0    & 7   & 0    & 30  & 29   & 98   & 147  & 384\\
521   & 1   & 0    & 7   & 8    & 30  & 53   & 100  & 215  & 564\\
4211  & 0   & 1    & 0   & 6    & 0   & 23   & 0    & 70   & 186\\
531   & 0   & 0    & 1   & 0    & 8   & 8    & 38   & 61   & 206\\
5211  & 0   & 0    & 0   & 1    & 0   & 7    & 0    & 30   & 100\\
532   & 0   & 0    & 0   & 0    & 1   & 0    & 9    & 8    & 39\\
5311  & 0   & 0    & 0   & 0    & 0   & 1    & 0    & 8    & 38\\
533   & 0   & 0    & 0   & 0    & 0   & 0    & 1    & 0    & 0\\
5321  & 0   & 0    & 0   & 0    & 0   & 0    & 0    & 1    & 9\\
5322  & 0   & 0    & 0   & 0    & 0   & 0    & 0    & 0    & 1\\
\hline
\end{tabular}
}
\caption{
Chern-Mather classes for all but $5$ Schubert varieties in the Cayley plane, the cominuscule space $E_6/P_6$.
}
\label{TBL:cMaE6}
\end{table}

\end{document}